\documentclass[final,leqno]{siamltex}

\usepackage{amsfonts,amssymb} %%%add
\usepackage{amsmath} %%%add
\usepackage{color} %%%add
\usepackage{graphicx} %%%add
\usepackage{epsfig,mathrsfs} %%%add
\usepackage[bf,SL,BF]{subfigure} %%%add
\usepackage{fancyhdr} %%%add
\usepackage{CJK} %%%add
\usepackage{caption} %%%add
\usepackage{wrapfig} %%%add
\usepackage{cases} %%%add
\usepackage{bm}%%%add
\usepackage{algorithm}
\usepackage{algorithmic}
\newtheorem{remark}{Remark}[section] %%%add
\newtheorem{example}{Example}[section] %%%add

\title{High order algorithms for the fractional substantial diffusion equation with truncated L\'{e}vy flights
\thanks{This work was supported by the National Natural Science Foundation of China under
Grant No. 11271173}}

\author{Minghua Chen\thanks{ School of Mathematics and Statistics, Gansu Key Laboratory of Applied Mathematics and Complex Systems, Lanzhou University, Lanzhou 730000, P.R. China (E-mail: chenmh09@163.com) }
        \and Weihua Deng\thanks{Corresponding author (dengwh@lzu.edu.cn). School of Mathematics and Statistics,
Gansu Key Laboratory of Applied Mathematics and Complex Systems, Lanzhou University, Lanzhou 730000, P.R. China   }}

\begin{document}

\maketitle

\begin{abstract}
The equation with the time fractional substantial derivative and space fractional derivative describes the distribution of the functionals of the  L\'{e}vy flights; and the equation is derived as the macroscopic limit of the continuous time random walk in unbounded domain and the L\'{e}vy flights have divergent second order moments. However, in more practical problems, the physical domain is bounded and the involved observables have finite moments. Then the modified equation can be derived by tempering the probability of large jump length of the L\'{e}vy flights and the corresponding tempered space fractional derivative is introduced. This paper focuses on providing the high order algorithms for the modified equation, i.e., the equation with the time fractional substantial derivative and space tempered  fractional derivative. More concretely, the contributions of this paper are as follows: 1. the detailed numerical stability analysis and error estimates of the schemes with  first order accuracy in time and second order in space are given in {\textsl{complex}} space, which is necessary since the inverse Fourier transform needs to be made for getting the distribution of the functionals after solving the equation; 2. we further propose the schemes with high order accuracy in both time and space, and the techniques of treating the issue of keeping the high order accuracy of the schemes for  {\textsl{nonhomogeneous}} boundary/initial conditions are introduced; 3. the multigrid methods are effectively used to solve the obtained algebraic equations which still have the Toeplitz structure; 4. we perform extensive numerical experiments, including verifying the high convergence orders, simulating the physical system which needs to numerically make the inverse Fourier transform to the numerical solutions of the equation.

\end{abstract}

\begin{keywords}
fractional substantial diffusion equation, truncated L\'{e}vy flights, high order algorithm, multigrid method, inverse discrete Fourier transform,
numerical stability and convergence
\end{keywords}

\begin{AMS}
26A33, 65M06, 65M12, 65M55, 65T50
\end{AMS}

\pagestyle{myheadings}
\thispagestyle{plain}
\markboth{M. H. CHEN AND W. H. DENG}{ALGORITHMS FOR FRACTIONAL SUBSTANTIAL EQUATION}

\section{Introduction}
Nowadays, more and more anomalous diffusion phenomena are found in nature, from the motion of mRNA molecules in live cells to the flight of albatrosses. Anomalous diffusion is a stochastic process with nonlinear relationship to time, in contrast to the normal diffusion. The continuous time random walk (CTRW) with the power law waiting time and/or jump length can microscopically describes anomalous diffusion; its macroscopic limit in unbounded domain leads to the time and/or space fractional diffusion equation \cite{Barkai:00}. In fact, based on the fractional Fourier law and conservation law, the fractional diffusion equation can also be easily derived. The space fractional diffusion equation defined in unbounded domain characterizes the probability distribution of the positions of the particles of the L\'{e}vy flights, which has the divergent second order moment. The more practical transport problems often take place in bounded domain and the involved observables have finite second moments; exponentially tempering the probability of large jumps of the L\'{e}vy flights, i.e.,  making the L\'{e}vy density decay as $e^{-\lambda |x|}|x|^{-1-\alpha}$ with $\lambda>0$ \cite{del-Castillo-Negrete:09,Koponen:94} leads to the space tempered fractional diffusion equation \cite{Baeumera:10,Cartea:07,del-Castillo-Negrete:09,Hanert:14,Sabzikar:14}; in other words, the space tempered fractional diffusion equation describes the probability density function (PDF) of the truncated L\'{e}vy flights.

The functionals of the trajectories of the particles often attract practical interest, e.g., the time spent by a particle in a given domain, the macroscopic measured signal in the nuclear magnetic resonance experiments; and the functionals are usually defined as  $ A=\int_0^t U[x(\tau)]d\tau$. When $x(t)$ is the trajectory of a Brownian particle, Kac derived a Schr\"{o}dinger-like
equation for the distribution of the functionals of diffusive motion; with the deep understanding of the anomalous diffusion, the distribution of the functionals of the paths of anomalous diffusion naturally attracts the interests of scientists, and the corresponding fractional Feynman-Kac equation is derived \cite{Carmi:10,Carmi:11,Turgeman:09}, which involves the fractional substantial derivative \cite{Friedrich:06}. As mentioned in the first paragraph, sometimes because of the bounded physical domain and the finite second moments of the observables, we need to temper the probability of the large jump length of the L\'{e}vy flights and the tempered space fractional diffusion can be derived; if further choosing the waiting time distribution as power law $t^{-(1+\gamma)}$ with $0<\gamma<1$, then the diffusion equation with time fractional derivative and space tempered fractional derivative appears. How about the distribution of the functionals of the paths of the particles with the jump length distribution $e^{-\lambda |x|}|x|^{-1-\alpha}$ and waiting time distribution $t^{-(1+\gamma)}$? Following the idea of \cite{Carmi:10}, we can easily derive the equation (\ref{1.1}) with the time fractional substantial derivative and space tempered fractional derivative to describe it. This paper focuses on providing effective and high accurate numerical algorithms for the new equation with homogeneous or nonhomogeneous boundary conditions, being given as
\begin{equation} \label{1.1}
\left\{ \begin{array}
 {l@{\quad} l}
\displaystyle {_c^sD}_t^\gamma G(x,\rho,t)= K {\nabla}_x^{\alpha,\lambda} G(x,\rho,t);\\
G(x,\rho,0)=G_0(x,\rho),~~~~x\in  (a,b);\\
G(a,\rho,t)=B(a,\rho,t),~~~~G(b,\rho,t)=B(b,\rho,t), ~~~~0<t \leq T.\\
\end{array}
 \right.
\end{equation}
Here, $\rho$ is the Fourier pair of $A$ and $G(x,A,t)$ is the joint PDF of finding the particle at time $t$ with the functional value $A$ and the initial position of the particle at $x$. The diffusion coefficient $K$ is a positive constant.  The Caputo fractional substantial derivative ${_c^sD}_t^\gamma$ with $\gamma \in (0,1)$ and
the Riesz tempered fractional derivative  ${\nabla}_x^{\alpha,\lambda}$  with $\lambda>0$, $\alpha \in (1,2)$ are, respectively, defined in the following:
\begin{definition}(\cite{Chen:13})\label{definition1.2}
Let $ \gamma>0$, $\rho$ be a real number, and $G(\cdot,t)$ be (m-1)-times continuously differentiable on $(0,\infty)$ and its $m$-times derivative be integrable on any finite subinterval of  $[0,\infty)$, where $m$ is the smallest integer that exceeds $\gamma$. Then
the Caputo fractional substantial derivative of $G(\cdot,t)$ of order $\gamma$ is defined by
\begin{equation*}
 {_c^sD}_t^\gamma G(x,\rho,t)={^s\!}I_t^{m-\gamma}[{^s\!}D_t^m G(x,\rho,t)],
\end{equation*}
where ${^s\!} D_t^m=\left(\frac{\partial}{\partial t}-J\rho  U(x)\right)^m$, $J=\sqrt{-1}$, and $U(x)$ is a prescribed  function. And the fractional substantial integral ${^s\!}I_t^\beta $ with $\beta>0$ is defined as
\begin{equation*}
{^s\!}I_t^\beta G(x,\rho,t)=\frac{1}{\Gamma(\beta)}\int_{0}^t{\left(t-\tau\right)^{\beta-1}}e^{J\rho U(x)(t-\tau)}{G(x,\rho,\tau)}d\tau, ~~~~t>0.
\end{equation*}
In fact, the Riemann-Liouville fractional substantial derivative is defined by
\begin{equation*}
{^s\!}D_t^\gamma G(x,\rho,t)={^s\!}D_t^m[{^s\!}I_t^{m-\gamma} G(x,\rho,t)].
\end{equation*}
\end{definition}

The Riesz tempered (truncated)  fractional derivative   is defined by  \cite{Cartea:07}
\begin{equation}\label{1.2}
\nabla_x^{\alpha,\lambda} G(x,\rho,t) =-\kappa_{\alpha}\left[ _{a}D_x^{\alpha,\lambda}+ _{x}\!D_{b}^{\alpha,\lambda} \right]G(x,\rho,t),
\end{equation}
where $\lambda>0, \,\kappa_{\alpha}=\frac{1}{2\cos(\alpha \pi/2)}$, $\alpha \in (1,2)$. The left and right Riemann-Liouville tempered fractional
derivatives are used,  being, respectively, defined by \cite{Baeumera:10,Cartea:07}
\begin{equation}\label{1.3}
\begin{split}
 _{a}D_x^{\alpha,\lambda}G(x,\rho,t)
&=e^{-\lambda x}{ _{a}}D_x^{\alpha}[e^{\lambda x}G(x,\rho,t)]-\lambda^\alpha G(x,\rho,t)
 -\alpha \lambda^{\alpha-1} \frac{\partial G(x,\rho,t) }{\partial x}\\
%&={_a^1}D_x^{\alpha,\lambda}G(x,\rho,t)-\lambda^\alpha G(x,\rho,t)
% -\alpha \lambda^{\alpha-1} \frac{\partial G(x,\rho,t) }{\partial x}\\
&={_a}\nabla_x^{\alpha,\lambda} G(x,\rho,t)-\alpha \lambda^{\alpha-1} \frac{\partial G(x,\rho,t) }{\partial x},
\end{split}
\end{equation}
and
\begin{equation}\label{1.4}
\begin{split}
 _{x}D_{b}^{\alpha,\lambda}G(x,\rho,t)
&=e^{\lambda x}{_{x}}D_{b}^{\alpha}[e^{-\lambda x}G(x,\rho,t)]-\lambda^\alpha G(x,\rho,t)
 +\alpha \lambda^{\alpha-1} \frac{\partial G(x,\rho,t) }{\partial x}\\
%&= {_x^1}D_{b}^{\alpha,\lambda}G(x,\rho,t)-\lambda^\alpha G(x,\rho,t)
% +\alpha \lambda^{\alpha-1} \frac{\partial G(x,\rho,t) }{\partial x}\\
&={_x}\nabla_b^{\alpha,\lambda} G(x,\rho,t)+\alpha \lambda^{\alpha-1} \frac{\partial G(x,\rho,t) }{\partial x},
\end{split}
\end{equation}
where  ${ _{a}}D_x^{\alpha}$ and ${_{x}}D_{b}^{\alpha}$ are, respectively, the left and right Riemann-Liouville fractional
derivatives \cite{Podlubny:99}; and
\begin{equation}\label{1.5}
\begin{split}
{_a}\nabla_x^{\alpha,\lambda} G(x,\rho,t)
&=e^{-\lambda x}{ _{a}}D_x^{\alpha}[e^{\lambda x}G(x,\rho,t)]-\lambda^\alpha G(x,\rho,t);\\
{_x}\nabla_b^{\alpha,\lambda} G(x,\rho,t)
&=e^{\lambda x}{_{x}}D_{b}^{\alpha}[e^{-\lambda x}G(x,\rho,t)]-\lambda^\alpha G(x,\rho,t);
\end{split}
\end{equation}
\begin{equation}\label{1.6}
\begin{split}
{_a^1}D_x^{\alpha,\lambda}G(x,\rho,t)
&=e^{-\lambda x}{ _{a}}D_x^{\alpha}[e^{\lambda x}G(x,\rho,t)];\\
{_x^1}D_{b}^{\alpha,\lambda}G(x,\rho,t)
&=e^{\lambda x}{_{x}}D_{b}^{\alpha}[e^{-\lambda x}G(x,\rho,t)].
\end{split}
\end{equation}
 Then (\ref{1.2}) reduces to
\begin{equation}\label{1.7}
\nabla_x^{\alpha,\lambda} G(x,\rho,t) =-\kappa_{\alpha}\left[ _{a}\nabla_x^{\alpha,\lambda}+ {_x}\nabla_{b}^{\alpha,\lambda} \right]G(x,\rho,t).
\end{equation}

\begin{remark}\label{remark1.1}
For $\lambda=0$, there is no truncation and Eq. (\ref{1.1}) reduces to the backward fractional Feynman-Kac equation with L\'{e}vy flight   \cite{Carmi:10}.
For $\lambda=0$ and $\rho=0$, Eq. (\ref{1.1}) becomes the fractional Fokker-Planck equation \cite{Deng:08,Metzler:99} or  space-time Caputo-Riesz fractional diffusion equation \cite{Chen:12}.
\end{remark}

In recent years, some important progresses for the numerical algorithms of space and/or time fractional diffusion equations have been made, see, e.g.,  \cite{Li:09,Yang:11,Zeng:013,Zhang:12}. Based on the weighted and shifted Gr\"{u}nwald first order discretization \cite{Meerschaert:04}, Lubich second and third order discretizations \cite{Lubich:86}, a series of effective high order discretizations for space fractional derivatives are recently derived \cite{Chen:1313,Chen:0013,Tian:12}. The extensions of the high order discretizations to fractional substantial derivative and tempered space fractional derivative can be found at \cite{Deng:14} and \cite{Li:14}. It should be noted that the obtained schemes have high order accuracy only when the boundary conditions are homogeneous or even the solution's high order derivatives at the boundaries are zero since the schemes are derived by using the Fourier transforms. One of the contributions of this paper is to present the techniques to overcome this limitation. We provide the detailed numerical stability and convergence analyses in the complex space for the scheme of (\ref{1.1}) with first order discretization in time and second order in space; performing the numerical analysis in complex space is necessary since to find the distribution of the functionals we need to make the inverse Fourier transform for the solutions of (\ref{1.1}). Then the schemes with high order accuracy in both time and space are proposed for (\ref{1.1}) with nonhomogeneous boundary and/or initial conditions, where the techniques of treating the nonhomogeneous boundary/initial conditions are introduced. All the schemes still have the potential Toeplitz structure and the multigrid methods are effectively used to solve the obtained algebraic equations; the convergence orders of the schemes are numerically confirmed. The high order schemes of tempered fractional derivatives can keep the same computation cost with first order scheme but greatly improve the accuracy \cite{Chen:0013}, since the nonlocal properties of fractional operators.  By numerically making the inverse Fourier transform for a series of solutions of (\ref{1.1}) with different $\rho$, we simulate the physical system.

The organization of the rest of the paper is as follows. Section $2$ is composed of $4$ subsections: the first subsection provides the effective second order space discretization for the tempered fractional derivative; the second subsection presents the numerical schemes of (\ref{1.1}) with $\nu$-th  order accuracy in time and second order in space, where $\nu=1,2,3,4$; the detailed numerical stability and convergence analyses for the scheme with first order accuracy in time and second order in space are given in the third subsection; in the fourth subsection, the numerical experiments are performed by multigrid methods to verify the convergence orders. Section 3 focuses on proposing the high order schemes of (\ref{1.1}) with nonhomogeneous boundary and/or initial conditions; the convergence orders are confirmed by numerical experiments and the physical systems are simulated.  Finally,  we conclude the paper with some remarks in the last section.

\section{High order schemes for (\ref{1.1}) with homogeneous boundary conditions}  This section focuses on discussing the high order schemes of (\ref{1.1}) with homogeneous boundary conditions and it is divided into four subsections. In \S 2.1, we derive the effective second order discretizations in space; and the schemes with high accuracy in both time and space are given in \S 2.2. The detailed proofs of stability and convergence for the scheme with first order accuracy in time and second order in space are presented in \S 2.3; and the last subsection numerically confirms the convergence orders by multigrid methods. Let us first introduce the following lemmas.

\begin{lemma} \label{lemma2.1}
Let $a=-\infty$ and $b=\infty$ in  (\ref{1.6}), respectively. Then we have
\begin{equation*}
e^{-\lambda x}{_{-\infty}}D_{x}^{\alpha}[e^{\lambda x}G(x)]=\frac{1}{\Gamma(n-\alpha)}
\left (  \frac{d}{dx}+\lambda\right)^n \int_{-\infty}^x \frac{e^{-\lambda (x-\xi)}G(\xi)}{(x-\xi)^{\alpha-n+1}}d\xi;
\end{equation*}
and
\begin{equation*}
e^{\lambda x}{_x}D_{\infty}^{\alpha}[e^{-\lambda x}G(x)]= \frac{(-1)^n}{\Gamma(n-\alpha)}
\left (  \frac{d}{dx}-\lambda\right)^n \int_{x}^{\infty} \frac{e^{-\lambda (\xi-x)}G(\xi)}{(\xi-x)^{\alpha-n+1}}d\xi.
\end{equation*}
\end{lemma}
\begin{proof}
The first equation of this lemma can be seen in Remark 2 of \cite{Li:14}. Here, we mainly prove the second one.
From the equality
$$\frac{d^n}{dx^n} \left[e^{-\lambda x} G(x) \right]=e^{-\lambda x} \left(\frac{d}{dx}-\lambda\right)^n   G(x),
$$
we have
\begin{equation*}
\begin{split}
e^{\lambda x}{_x}D_{\infty}^{\alpha}[e^{-\lambda x}G(x)]
&=e^{\lambda x}  (-1)^n\frac{d^n}{dx^n}\left\{{_x}D_{\infty}^{\alpha-n}[e^{-\lambda x}G(x)]\right \}\\
&=e^{\lambda x}  (-1)^n\frac{d^n}{dx^n}\left[e^{-\lambda x}\frac{1}{\Gamma(n-\alpha)}
 \int_{x}^{\infty} \frac{e^{-\lambda (\xi-x)}G(\xi)}{(\xi-x)^{\alpha-n+1}}d\xi\right ]\\
%&=e^{\lambda x}  (-1)^n  e^{-\lambda x} \left(\frac{d}{dx}-\lambda\right)^n  \frac{1}{\Gamma(n-\alpha)}
% \int_{x}^{\infty} \frac{e^{-\lambda (\xi-x)}G(\xi)}{(\xi-x)^{\alpha-n+1}}d\xi\\
&=\frac{(-1)^n}{\Gamma(n-\alpha)}
\left (  \frac{d}{dx}-\lambda\right)^n \int_{x}^{\infty} \frac{e^{-\lambda (\xi-x)}G(\xi)}{(\xi-x)^{\alpha-n+1}}d\xi.
\end{split}
\end{equation*}
The proof is completed.
\end{proof}
\begin{remark}\label{remark2.1}(\cite{Chen:13,Li:14}) Based on Lemma \ref{lemma2.1}, it can be noted that the fractional substantial calculus and
the first term of the tempered (truncated) fractional operators (corresponding to   (\ref{1.6}) ) have almost the same mathematical formulations.
\end{remark}

\begin{lemma} (\cite{Cartea:07,Chen:13,Li:14})\label{lemma2.2}
 Let $\alpha>0$, $G \in C_0^{\infty}(\Omega)$, $\Omega \subset \mathbb{R}$. Then
\begin{equation*}
\mathcal{F}\left( e^{-\lambda x}{ _{-\infty}}D_x^{\alpha}[e^{\lambda x}G(x)]\right)
 =(\lambda-i\omega)^{\alpha}\widehat{G}(\omega);
\end{equation*}
and
\begin{equation*}
\mathcal{F}\left( e^{\lambda x}{ _x}D_{\infty}^{\alpha}[e^{-\lambda x}G(x)]\right)
  =(\lambda+i\omega)^{\alpha}\widehat{G}(\omega),
\end{equation*}
where $\mathcal{F}$ denotes Fourier transform operator and $\widehat{G}(\omega)=\mathcal{F}(G)$, i.e.,
\begin{equation*}
    \widehat{G}(\omega)=\int_{\mathbb{R}}e^{i\omega x }G(x)dx.
 \end{equation*}
\end{lemma}

\subsection{Derivation of the effective second order discretizations in space}
From Lemma \ref{lemma2.1}, Remark \ref{remark2.1} and \cite{Chen:13}, it follows that  the $L$-th order ($L\leq 5$)  approximations of the $\alpha$-th left Riemann-Liouville tempered fractional derivative can be generated by the corresponding coefficients of the generating functions $\kappa ^{L,\alpha}(\zeta)$,
\begin{equation*}
\kappa^{L,\alpha}(\zeta) = \left(\sum_{j=1}^L\frac{1}{j}\left(1- e^{-\lambda h} \zeta  \right)^j\right)^{\alpha},
\end{equation*}
where $h=(b-a)/M$ is the uniform stepsize and
$x_i=a+ih$, $i=0,1,\ldots,M-1,M$.

Taking $L=1$,  the above equation can be recast for all $|\zeta| \leq 1$ as
\begin{equation}\label{2.1}
\begin{split}
&\left(1-e^{-\lambda h}\zeta\right)^{\alpha}
= \sum_{j=0}^{\infty}e^{-j\lambda h}g_j^{\alpha}\zeta^{j}
\end{split}
\end{equation}
with
\begin{equation}\label{2.2}
\begin{split}
& g_0^{\alpha}=1, ~~~~g_j^{\alpha}=\left(1-\frac{\alpha+1}{j}\right)g_{j-1}^{\alpha},~~j \geq 1.
 \end{split}
\end{equation}

\begin{lemma}[\cite{Baeumera:10}]\label{lemma2.3}
Let $G $, $_{-\infty}\nabla_x^{\alpha+1,\lambda}G(x)$  and their Fourier transforms belong to $L_1(\mathbb{R})$; and $p \in \mathbb{R}$, $\lambda \geq 0$, $h>0$, $\alpha \in (1,2)$.  Define
 \begin{equation}\label{2.3}
A_{p}^{\alpha,\lambda}G(x)=\frac{1}{h^{\alpha}}\sum_{j=0}^{\infty}e^{-(j-p)\lambda h}g_j^{\alpha}G(x-(j-p)h)
                            -e^{ph\lambda}\frac{(1-e^{-h\lambda})^\alpha}{h^\alpha}G(x),
\end{equation}
where $g_j^{\alpha}$ is defined by (\ref{2.2}). Then
 $$_{-\infty}\nabla_x^{\alpha,\lambda}G(x)=A_{p}^{\alpha,\lambda}G(x)+\mathcal{O}(h).$$
\end{lemma}
\begin{proof}
For bringing convenience to the analysis of the following high order schemes, we reprove this lemma; the way of proof is almost the same as the one of \cite{Baeumera:10}.
\begin{equation*}
\begin{split}
\mathcal{F}(A_{p}^{\alpha,\lambda}G)(\omega)
&=\frac{1}{h^{\alpha}}\sum_{j=0}^{\infty}e^{-(j-p)\lambda h}g_j^{\alpha}e^{i\omega(j-p) h}\widehat{G}(\omega)
                            -e^{ph\lambda}\frac{(1-e^{-h\lambda})^\alpha}{h^\alpha}\widehat{G}(\omega)\\
&=(\lambda-i\omega)^{\alpha}\omega_{p,h}(\lambda-i\omega)\widehat{G}(\omega)-\lambda^\alpha \omega_{p,h}(\lambda)\widehat{G}(\omega)
\end{split}
\end{equation*}
with
\begin{equation}\label{2.4}
\omega_{p,h}(z)=e^{ phz} \left(\frac{1-e^{-hz}}{hz}\right )^{\alpha}=1+\left(p-\frac{\alpha}{2}\right)zh+ \mathcal{O}(h^2).
\end{equation}
Therefore, from Lemma \ref{lemma2.2},  there exists
\begin{equation*}
\begin{split}
\mathcal{F}(A_p^{\alpha,\lambda}G)(\omega)=\mathcal{F}(_{-\infty}\nabla_x^{\alpha,\lambda}G)(\omega)+ \widehat{\phi}(\omega),
  \end{split}
\end{equation*}
where $ \widehat{\phi}(\omega)
=\left(p-\frac{\alpha}{2}\right)\left[(\lambda -i\omega)^{\alpha+1}-\lambda^{\alpha+1}\right]\widehat{f}(\omega)\cdot h+\mathcal{O}(h^2)$.
 Then
\begin{equation*}
\begin{split}
&|\widehat{\phi}(\omega)| \leq \widetilde{c}\cdot\left(|(\lambda -i\omega)^{\alpha+1}\widehat{f}(\omega)|+|\lambda^{\alpha+1}\widehat{f}(\omega)|\right)\cdot h.
\end{split}
\end{equation*}
With the condition  $\mathcal{F}[\mathcal{_{-\infty}}\nabla_x^{\alpha+1,\lambda}G(x)] \in L_1(\mathbb{R})$, it leads to
\begin{equation*}
\begin{split}
| \mathcal{_{-\infty}}\nabla_x^{\alpha,\lambda}G(x)-A_p^{\alpha,\lambda}G(x)|
=|\phi(x)| \leq \frac{1}{2\pi}\int_{\mathbb{R}}|\widehat{\phi}(\omega)|dx=\mathcal{O}(h).
  \end{split}
\end{equation*}

\end{proof}

\begin{lemma}\label{lemma2.4}
Let $G $, $_{-\infty}\nabla_x^{\alpha+2,\lambda}G(x)$ and their Fourier transforms belong to $L_1(\mathbb{R})$; and $p \in \mathbb{R}$, $\lambda \geq 0$, $h>0$, $\alpha \in (1,2)$.  Define
 \begin{equation}\label{2.5}
A_{r_1,r_2,r_3}^{\alpha,\lambda}G(x)=\left(r_1A_{1}^{\alpha,\lambda}+r_2A_{0}^{\alpha,\lambda}+r_3A_{-1}^{\alpha,\lambda}\right)G(x),
\end{equation}
where $A_{1}^{\alpha,\lambda}$, $A_{0}^{\alpha,\lambda}$ and $A_{-1}^{\alpha,\lambda}$ are defined by (\ref{2.3}). Then
 $$_{-\infty}\nabla_x^{\alpha,\lambda}G(x)=A_{r_1,r_2,r_3}^{\alpha,\lambda}G(x)+\mathcal{O}(h^2),$$
where
\begin{equation}\label{2.6}
r_1+r_2+r_3=1;~~ r_1=\frac{\alpha}{2}+r_3,~~r_2=\frac{2-\alpha}{2}-2r_3,~~ \forall ~r_3.
\end{equation}
\end{lemma}
\begin{proof}
From Lemma \ref{lemma2.3}, it leads to
\begin{equation*}
\begin{split}
&\mathcal{F}(A_{r_1,r_2,r_3}^{\alpha,\lambda}G)(\omega)\\
&=(\lambda-i\omega)^{\alpha}\left[r_1\omega_{1,h}(\lambda-i\omega)+r_2\omega_{0,h}(\lambda-i\omega)+r_{3}\omega_{-1,h}(\lambda-i\omega)\right]\widehat{G}(\omega)\\
&\quad-\lambda^\alpha \left[r_1\omega_{1,h}(\lambda)+r_2\omega_{0,h}(\lambda)+r_{3}\omega_{-1,h}(\lambda)\right]\widehat{G}(\omega),
\end{split}
\end{equation*}
and using (\ref{2.4}), we have
$$r_1\omega_{1,h}(z)+r_2\omega_{0,h}(z)+r_{3}\omega_{-1,h}(z)=1+ \mathcal{O}(h^2).$$
Therefore, from Lemma \ref{lemma2.2},  we obtain
\begin{equation*}
\begin{split}
\mathcal{F}(A_{r_1,r_2,r_3}^{\alpha,\lambda}G)(\omega)=\mathcal{F}(_{-\infty}\nabla_x^{\alpha,\lambda}G)(\omega)+ \widehat{\phi}(\omega)
  \end{split}
\end{equation*}
with $|\widehat{\phi}(\omega)|=\mathcal{O}(h^2)$. It yields
\begin{equation*}
\begin{split}
| \mathcal{_{-\infty}}\nabla_x^{\alpha,\lambda}G(x)-A_{r_1,r_2,r_3}^{\alpha,\lambda}G(x)|
=|\phi(x)| \leq \frac{1}{2\pi}\int_{\mathbb{R}}|\widehat{\phi}(\omega)|dx=\mathcal{O}(h^2).
  \end{split}
\end{equation*}

\end{proof}

Assume that the well-defined function $G(x)$ can be zero extended from the bounded domain $(a, b)$ to $(-\infty, b)$, and satisfy the requirements of the above corresponding theorems; and
 \begin{equation}\label{2.7}
  \begin{split}
\widetilde{A}_{p}^{\alpha,\lambda}G(x_i)=\frac{1}{h^{\alpha}}\sum_{j=0}^{i+p}e^{-(j-p)\lambda h}g_j^{\alpha}G(x_{i-j+p})
 -e^{ph\lambda}\frac{(1-e^{-h\lambda})^\alpha}{h^\alpha}G(x_i).
\end{split}
\end{equation}
Then
 \begin{equation}\label{2.8}
 \begin{split}
_{a}\nabla_{x}^{\alpha,\lambda}G(x_i)=\widetilde{A}_{p}^{\alpha,\lambda}G(x_i)+\mathcal{O}(h);
 \end{split}
\end{equation}
and
\begin{equation}\label{2.9}
 \begin{split}
_{a}\nabla_{x}^{\alpha,\lambda}G(x_i)
&=\widetilde{A}_{r_1,r_2,r_3}^{\alpha,\lambda}G(x_i)+\mathcal{O}(h^2)\\
%&=r_1\widetilde{A}_{1}^{\alpha,\lambda}G(x_i)+r_2\widetilde{A}_{0}^{\alpha,\lambda}G(x_i)
%  +r_3\widetilde{A}_{-1}^{\alpha,\lambda}G(x_i)+\mathcal{O}(h^2)\\
%&=\frac{1}{h^{\alpha}}\left[\sum_{j=0}^{i+1}e^{-(j-1)\lambda h}r_1g_j^{\alpha}G(x_{i-j+1})
%+\sum_{j=0}^{i}e^{-j\lambda h}r_2g_j^{\alpha}G(x_{i-j})\right.\\
%&\quad\left.+\sum_{j=0}^{i-1}e^{-(j+1)\lambda h}r_3g_j^{\alpha}G(x_{i-j-1})\right]+\mathcal{O}(h^2)\\
&=\frac{1}{h^{\alpha}}\sum_{j=0}^{i+1}e^{-(j-1)\lambda h}\omega_j^{\alpha}G(x_{i-j+1})+\mathcal{O}(h^2).
 \end{split}
\end{equation}
Here
\begin{equation}\label{2.10}
 \begin{split}
&\widetilde{A}_{r_1,r_2,r_3}^{\alpha,\lambda}G(x_i)
=\left(r_1\widetilde{A}_{1}^{\alpha,\lambda}+r_2\widetilde{A}_{0}^{\alpha,\lambda}
  +r_3\widetilde{A}_{-1}^{\alpha,\lambda}\right)G(x_i);
 \end{split}
\end{equation}
and
\begin{equation}\label{2.11}
\begin{split}
&\omega_{0}^\alpha=r_1g_{0}^\alpha; ~~~~
\omega_{1}^\alpha=r_1g_{1}^\alpha+r_2g_{0}^\alpha-\left(r_1e^{\lambda h}+r_2+r_3e^{-\lambda h}\right)\left(1-e^{-\lambda h}\right)^{\alpha};\\
&\omega_{j}^\alpha=r_1g_{j}^\alpha+r_2g_{j-1}^\alpha+r_3g_{j-2}^\alpha, ~~~~2\leq j \leq M-1.
\end{split}
\end{equation}

\begin{remark} \label{remark2.2}(\cite{Tian:12})
When employing the difference method with (\ref{2.7}) for approximating non-periodic boundary problems on bounded interval,
$p$ should be chosen satisfying $|p|\leq 1$ to ensure that the nodes at which the values of $G$ are within the bounded interval.
\end{remark}

Let $|p|\leq 1$ and  $\widetilde{G}=[G({x_1}),G({x_2}),\cdots,G({x_{M-1}})]^{\rm T}$.
Then (\ref{2.10}) can be rewritten as the following matrix form
\begin{equation}\label{2.12}
  \begin{split}
&\widetilde{A}_{r_1,r_2,r_3}^{\alpha,\lambda}\widetilde{G}=\frac{1}{h^{\alpha}}A^{\alpha,\lambda} \widetilde{G}
%=\frac{1}{h^{\alpha}}\Big[r_1{A}_{1}^{\alpha,\lambda}+r_2{A}_{0}^{\alpha,\lambda}+r_3{A}_{-1}^{\alpha,\lambda}\Big] \widetilde{G},
\end{split}
\end{equation}
with
\begin{equation}\label{2.13}
  A^{\alpha,\lambda}=\left [ \begin{matrix}
e^{0\lambda h}\omega_1^\alpha               &e^{\lambda h}\omega_{0}^\alpha             &                     \\
e^{-\lambda h}\omega_{2}^\alpha             &e^{0\lambda h}\omega_1^\alpha              &e^{\lambda h}\omega_{0}^\alpha  &                      \\
e^{-2\lambda h}\omega_{3}^\alpha            &e^{-\lambda h}\omega_{2}^\alpha            &e^{0\lambda h}\omega_1^\alpha   &\ddots        \\
\vdots                                 &  \cdots                              &    \ddots  & \ddots  &   \ddots                        \\
e^{-(M-3)\lambda h}\omega_{M-2}^\alpha  &\cdots                                &  \cdots    & \ddots  &e^{0\lambda h}\omega_{1}^\alpha  &e^{\lambda h}\omega_{0}^\alpha \\
e^{-(M-2)\lambda h}\omega_{M-1}^\alpha  &e^{-(M-3)\lambda h}\omega_{M-2}^\alpha &  \cdots    & \cdots  &e^{-\lambda h}\omega_{2}^\alpha  &e^{0\lambda h}\omega_{1}^\alpha
 \end{matrix}
 \right ].
\end{equation}

Similarly, for the right Riemann-Liouville tempered fractional derivative, the second order approximation is given as
\begin{equation}\label{2.14}
 \begin{split}
_{x}\nabla_{b}^{\alpha,\lambda}G(x_i)
&=\widetilde{B}_{r_1,r_2,r_3}^{\alpha,\lambda}G(x_i)+\mathcal{O}(h^2)\\
&=\frac{1}{h^{\alpha}}\sum_{j=0}^{M-i+1}e^{-(j-1)\lambda h}\omega_j^{\alpha}G(x_{i+j-1})+\mathcal{O}(h^2),
 \end{split}
\end{equation}
where  $\omega_j^{\alpha}$ is defined by (\ref{2.11}), and the matrix  form is
\begin{equation}\label{2.15}
  \begin{split}
&\widetilde{B}_{r_1,r_2,r_3}^{\alpha,\lambda}U=\frac{1}{h^{\alpha}}B^{\alpha,\lambda} U
~~{\rm with}~~B^{\alpha,\lambda}=(A^{\alpha,\lambda})^T.
\end{split}
\end{equation}

In the following, we focus on how to choose the parameters $r_3$ such that
all the eigenvalues of the matrix $A^{\alpha,\lambda}$
have negative real parts and $A^{\alpha,\lambda}+(A^{\alpha,\lambda})^T $ is diagonally dominant; this means that the corresponding schemes work for space tempered fractional derivatives. Firstly, we present several useful lemmas.
\begin{definition}\cite[p.\,27]{Quarteroni:07}\label{definition2.5}
A matrix $A \in \mathbb{C}^{n\times n}$ is positive definite in $\mathbb{C}^{n}$ if the real number  $(Ax,x)>0$ for all
$ x \in \mathbb{C}^{n}$, $x\neq 0$.
\end{definition}

\begin{lemma}\cite[p.\,27]{Quarteroni:07}\label{lemma2.6}
A square matrix $A$ of order $n$ is positive definite in  $\mathbb{C}^{n}$ if and only if it is hermitian and has positive eigenvalues.
\end{lemma}

\begin{lemma}\cite[p.\,28]{Quarteroni:07}\label{lemma2.7}
A real matrix $A$ of order $n$ is positive definite  if and only if  its symmetric part $H=\frac{A+A^T}{2}$ is positive definite.
\end{lemma}

\begin{lemma}\cite[p.\,184]{Quarteroni:07}\label{lemma2.8}
Let $A \in \mathbb{C}^{n \times n}$ and  $H=\frac{A+A^H}{2}$ be the hermitian part of $A$.  Then for any eigenvalue $\lambda$ of  $A$,
the real part $\Re(\lambda(A))$ satisfies
\begin{equation*}
  \lambda_{\min}(H) \leq \Re(\lambda(A)) \leq \lambda_{\max}(H),
\end{equation*}
where $\lambda_{\min}(H)$ and $\lambda_{\max}(H)$ are the minimum and maximum of the eigenvalues of $H$, respectively.
\end{lemma}

\begin{lemma}\label{lemma2.9}
Let $A^{\alpha,\lambda}$ be given in (\ref{2.13}) with $\lambda \geq 0$, $h>0$ and   $1<\alpha<2$. If $r_3$ in (\ref{2.6}) satisfies
\begin{equation}\label{2.16}
\max \left \{ -\frac{\alpha(\alpha-1)(\alpha+2)}{2(\alpha^2+3\alpha+4)}, -\frac{(2-\alpha)(8-\alpha^2-\alpha)}{2(\alpha+1)(\alpha+2)}  \right\}
\leq r_3 \leq \frac{(\alpha-1)(2-\alpha)(\alpha+3)}{2(\alpha+1)(\alpha+2)},
\end{equation}
then the elements, denoted by $\phi_j^\alpha$, of $H=A^{\alpha,\lambda}+(A^{\alpha,\lambda})^T$ defined by (\ref{2.18}) satisfy
\begin{equation}\label{2.17}
\phi_1^{\alpha}<0, \,\phi_0^{\alpha}+\phi_2^{\alpha}\geq 0, \, \phi_3^{\alpha}>0, ~~\rm{and}~~ \phi_j^{\alpha}>0 ~ {\rm when} ~ j\geq 4.
\end{equation}
\end{lemma}

\begin{proof}
Denote $H=A^{\alpha,\lambda}+(A^{\alpha,\lambda})^T$.  Using (\ref{2.13}) and (\ref{2.11}), we obtain
\begin{equation}\label{2.18}
H=\left [ \begin{matrix}
2\phi_1^{\alpha}   &\phi_0^{\alpha}+\phi_2^{\alpha}&\phi_3^{\alpha}     &      \cdots   &  \phi_{M-2}^{\alpha}     &  \phi_{M-1}^{\alpha}  \\
\phi_0^{\alpha}+\phi_2^{\alpha}&  2\phi_1^{\alpha}    &\phi_0^{\alpha}+\phi_2^{\alpha}&  \phi_3^{\alpha}    &     \cdots   &  \phi_{M-2}^{\alpha} \\
\phi_3^{\alpha}             &\phi_0^{\alpha}+\phi_2^{\alpha}&2\phi_1^{\alpha}         & \phi_0^{\alpha}+\phi_2^{\alpha}&     \ddots              & \vdots  \\
\vdots                   &          \ddots         &       \ddots            &        \ddots            &      \ddots             &  \phi_3^{\alpha}  \\
\phi_{M-2}^{\alpha}  &   \ddots         &       \ddots        &        \ddots            &   2\phi_1^{\alpha}         & \phi_0^{\alpha}+\phi_2^{\alpha} \\
\phi_{M-1}^{\alpha}   &    \phi_{M-2}^{\alpha}   &   \cdots          &         \cdots           &\phi_0^{\alpha}+\phi_2^{\alpha}& 2\phi_1^{\alpha}
 \end{matrix}
 \right ],
\end{equation}
where
\begin{equation}\label{2.19}
\phi_0^{\alpha}=e^{\lambda h}\omega_0^\alpha; ~~\phi_1^{\alpha}=e^{0\lambda h}\omega_1^\alpha;
~~\phi_2^{\alpha}=e^{-\lambda h}\omega_2^\alpha;~~\phi_j^{\alpha}=e^{-(j-1)\lambda h}\omega_j^\alpha,~~j \geq 3;
\end{equation}
and
\begin{equation*}
\begin{split}
&\omega_{0}^\alpha=r_1g_{0}^\alpha;
~~~~\omega_{1}^\alpha=r_1g_{1}^\alpha+r_2g_{0}^\alpha-\left(r_1e^{\lambda h}+r_2+r_3e^{-\lambda h}\right)\left(1-e^{-\lambda h}\right)^{\alpha};\\
&\omega_{j}^\alpha=r_1g_{j}^\alpha+r_2g_{j-1}^\alpha+r_3g_{j-2}^\alpha, ~~~~j \geq 2.
\end{split}
\end{equation*}
%If taking $r_1\geq 0$ and $r_2\geq 0$ in (\ref{2.6}), then $r_3$ satisfies
%\begin{equation*}
%-\frac{\alpha}{2}\leq r_3\leq \frac{2-\alpha}{4}.
%\end{equation*}
Next we  prove that, under the requirement of $(\ref{2.16})$,  $\phi_1^{\alpha}<0$, $\phi_0^{\alpha}+\phi_2^{\alpha}\geq 0$,  $\phi_3^{\alpha}>0$, and $\phi_j^{\alpha}>0$, $j \geq 4$.

Case   $\phi_1^{\alpha}<0$: From (\ref{2.2}) and (\ref{2.6}), we can check that
$\phi_1^{\alpha}<0$ when $r_3>-\frac{\alpha-1}{2}$. And it can also be noted that in this case $r_1>0$.

Case   $\phi_0^{\alpha}+\phi_2^{\alpha}\geq 0$: Using $r_1\geq 0$, $\lambda \geq 0$, and $h>0$, we obtain
\begin{equation*}
\begin{split}
\phi_0^{\alpha}+\phi_2^{\alpha}
&=e^{\lambda h}\omega_0^\alpha+e^{-\lambda h}\omega_2^\alpha
= e^{\lambda h}r_1g_{0}^\alpha+e^{-\lambda h}\left(r_1g_{2}^\alpha+r_2g_{1}^\alpha+r_3g_{0}^\alpha\right) \\
&\geq  e^{-\lambda h}\left(r_1g_{0}^\alpha+r_1g_{2}^\alpha+r_2g_{1}^\alpha+r_3g_{0}^\alpha\right) \\
&=e^{-\lambda h}\left[r_3\frac{\alpha^2+3\alpha+4}{2}
 +\frac{\alpha(\alpha-1)(\alpha+2)}{4}\right] \geq 0,
\end{split}
\end{equation*}
which leads to  $r_3 \geq -\frac{\alpha(\alpha-1)(\alpha+2)}{2(\alpha^2+3\alpha+4)}.$

Case   $\phi_3^{\alpha}>0$: According to (\ref{2.2}), (\ref{2.6}), and (\ref{2.19}), we get
\begin{equation*}
\begin{split}
\phi_3^{\alpha}
&=e^{-2\lambda h}\omega_3^\alpha
= e^{-2\lambda h}\left(r_1g_{3}^\alpha+r_2g_{2}^\alpha+r_3g_{1}^\alpha\right) \\
&=e^{-2\lambda h}\left[r_3\frac{-\alpha(\alpha+1)(\alpha+2)}{6}
 +\frac{\alpha(\alpha-1)(2-\alpha)(\alpha+3)}{12}\right] \geq 0, \\
\end{split}
\end{equation*}
if and only if $ r_3 \leq \frac{(\alpha-1)(2-\alpha)(\alpha+3)}{2(\alpha+1)(\alpha+2)}.$

Case  $\phi_j^{\alpha}>0$, $j \geq 4$: From (\ref{2.2}), (\ref{2.6}) and (\ref{2.19}), there exist
\begin{equation*}
\begin{split}
\phi_j^{\alpha}
&=e^{-(j-1)\lambda h}\omega_j^\alpha=e^{-(j-1)\lambda h}\left(r_1g_{j}^\alpha+r_2g_{j-1}^\alpha+r_3g_{j-2}^\alpha\right) \\
&=e^{-(j-1)\lambda h}\left[r_3\frac{(\alpha+1)(\alpha+2) }{j(j-1)}  +\frac{(j-\alpha-2)(2j-\alpha^2-\alpha)}{2j(j-1)}  \right]g_{j-2}^\alpha \geq 0,\\
\end{split}
\end{equation*}
which results in $r_3 \geq -\frac{(2-\alpha)(8-\alpha^2-\alpha)}{2(\alpha+1)(\alpha+2)}$.
Then, the proof is completed.
\end{proof}

\begin{lemma}\label{lemma2.10}
Denote the elements of $H$ defined in (\ref{2.18}) by $h_{i,j}$ and let $r_3$ satisfy  (\ref{2.16}).
%
%Let $H$ in (\ref{2.18}) with elements $h_{i,j}$ and $r_3$ satisfies  (\ref{2.16}).
Then there exist
\begin{equation*}
  \begin{split}
&(1) ~~ h_{i,i}=2\phi_1^{\alpha}<0,~~~~ h_{i,j} > 0,~~(j\neq i);\\
&(2) ~~ \sum\limits_{j=0}^{\infty}h_{i,j}= 0 ~~ \mbox {and} ~~
    -h_{i,i}>\!\!\!\!\! \sum\limits_{j=0,j\neq i}^{M}\!\!\!\!h_{i,j}.
  \end{split}
\end{equation*}
\end{lemma}
\begin{proof}
From Lemma \ref{lemma2.9}, we have $h_{i,i}=2\phi_1^{\alpha}<0$ and $ h_{i,j} > 0\,\,(j\neq i)$.
Denote the elements of $A^{\alpha,\lambda}$ given in (\ref{2.13}) by $a_{i,j}$. Then from (\ref{2.11}), we have
\begin{equation*}
\begin{split}
\sum_{j=-\infty}^{i+1}a_{i,j}
&=a_{i,i+1}+a_{i,i}+a_{i,i-1}+\cdots+a_{i,i-j+1}+\cdots\\
&=e^{\lambda h}\omega_{0}^\alpha +e^{0\lambda h}\omega_{1}^\alpha+e^{-\lambda h}\omega_{2}^\alpha+e^{-2\lambda h}\omega_{3}^\alpha
+\cdots+e^{-(i-j)\lambda h}\omega_{i-j+1}^\alpha +\cdots \\
&=\sum_{j=0}^{\infty}e^{-(j-1)\lambda h}\omega_{j}^\alpha\\
&=\left(r_1e^{\lambda h}+r_2+r_3e^{-\lambda h}\right)\sum_{j=0}^{\infty}e^{-j\lambda h}g_j^{\alpha}
-\left(r_1e^{\lambda h}+r_2+r_3e^{-\lambda h}\right)\left(1-e^{-\lambda h}\right)^{\alpha}\\
&=0,
\end{split}
\end{equation*}
where $\sum\limits_{j=0}^{\infty}e^{-j\lambda h}g_j^{\alpha}=\left(1-e^{-\lambda h}\right)^{\alpha}$ obtained from (\ref{2.1}) is used.

Denote the elements of $(A^{\alpha,\lambda})^T$  by $b_{i,j}$. In a similar way, we have
\begin{equation*}
\begin{split}
\sum_{j=i-1}^{\infty}b_{i,j}=0.
\end{split}
\end{equation*}
Since $H=A^{\alpha,\lambda}+(A^{\alpha,\lambda})^T$,
it leads to
$$\sum_{j=-\infty}^{\infty}h_{i,j}=\sum_{j=-\infty}^{i+1}a_{i,j}+\sum_{j=i-1}^{\infty}b_{i,j}=0.$$
The proof is completed.
\end{proof}

\begin{theorem}\label{theorem2.11}
Let $A^{\alpha,\lambda}$ be given in (\ref{2.13}) with $\lambda \geq 0$, $h>0$,  $1<\alpha<2$, and $r_3$ in (\ref{2.16}).
Then any eigenvalue $\lambda$ of $A^{\alpha,\lambda}$ satisfies
 $$\Re(A^{\alpha,\lambda})<0.$$
Moreover, the matrixes $A^{\alpha,\lambda}$ and $(A^{\alpha,\lambda})^T$ are negative definite in $ \mathbb{R}^{n \times n}$.
%and $A^{\alpha,\lambda}+(A^{\alpha,\lambda})^T$ is negative definite in $ \mathbb{C}^{n \times n}$.
\end{theorem}

\begin{proof}
Firstly, we prove that all the eigenvalues of the matrix $H$ are negative. From Lemma \ref{2.9} we obtain
\begin{equation*}
  \Lambda_i= \sum\limits_{j=0,j\neq i}^{M}\!\!\!\!h_{i,j}<-h_{i,i}.
\end{equation*}
According to the Gerschgorin theorem \cite{Isaacson:66}, the eigenvalues of the matrix $H$ are in the disks centered at $h_{i,i}$, with radius
$ \Lambda_i$, i.e.,  the eigenvalues $\lambda$ of the matrix $H$ satisfy
\begin{equation*}
  |\lambda -h_{i,i} | \leq \Lambda_i.
\end{equation*}
In addition, the matrix $H$ is symmetric, then all the eigenvalues of the matrix $H$ are negative. Using Lemma \ref{lemma2.6}, it follows that
$H=A^{\alpha,\lambda}+(A^{\alpha,\lambda})^T$ is negative definite in $ \mathbb{C}^{n \times n}$.
According to Lemmas \ref{lemma2.7} and \ref{lemma2.8}, we know that the matrixes $A^{\alpha,\lambda}$ and $(A^{\alpha,\lambda})^T$ are negative definite in $ \mathbb{R}^{n \times n}$.
\end{proof}

\subsection{Derivation of the numerical schemes}
Take the mesh points $x_i=a+ih,i=0,1,\ldots ,M$, and $t_n=n\tau,n=0,1,\ldots ,{N}$, where
 $h=(b-a)/M$, $\tau=T/N$ are  the uniform space stepsize and time steplength, respectively.
Denote $G_{i,\rho}^n$ as the numerical approximation to $G(x_i,\rho,t_n)$.

From \cite{Chen:13,Deng:14},  we know that the following equation holds
\begin{equation}\label{2.20}
\begin{split}
{^s_c}{D}_t^\gamma G(x,\rho,t)
&={^s\!}D_t^\gamma [G(x,\rho,t)-e^{J\rho U(x) t}G(x,\rho,0)],~~~~J=\sqrt{-1}.
\end{split}
\end{equation}
And the Riemann-Liouville fractional substantial derivative has the $\nu$-th order approximations, i.e.,
\begin{equation}\label{2.21}
\begin{split}
&{^s\!}D_t^\gamma G(x,\rho,t)|_{(x_i,t_n)}=\frac{1}{\tau^\gamma}\sum_{k=0}^{n}{d}_{i,k}^{\nu,\gamma}G(x_i,\rho,t_{n-k})+\mathcal{O}(\tau^\nu), ~~\nu=1,2,3,4;\\
& {^s\!}D_t^\gamma [e^{J\rho U(x) t}G(x,\rho,0)]_{(x_i,t_n)}=\frac{1}{\tau^\gamma}\sum_{k=0}^{n}d_{i,k}^{\nu,\gamma}e^{J\rho U_i (n-k)\tau}G(x_i,\rho,0)+\mathcal{O}(\tau^\nu)
\end{split}
\end{equation}
with
\begin{equation}\label{2.22}
d_{i,k}^{\nu,\gamma}=e^{J\rho U_i k \tau}{l}_k^{\nu,\gamma},~~U_i=U(x_i),~~\nu=1,2,3,4,
\end{equation}
where ${l}_k^{1,\gamma}$, ${l}_k^{2,\gamma}$, ${l}_k^{3,\gamma}$, ${l}_k^{4,\gamma}$, and ${l}_k^{5,\gamma}$
are defined by (2.2),  (2.4), (2.6), (2.8), and (2.10) in \cite{Chen:1313}, respectively.
In particular, when $\nu=1$, there exists
\begin{equation}\label{2.23}
  d_{i,k}^{1,\gamma}=e^{J\rho U_i k \tau}g_k^\gamma,~~ g_k^\gamma=(-1)^k\left ( \begin{matrix} \gamma \\ k\end{matrix} \right ).
\end{equation}

Combining (\ref{2.9}), (\ref{2.14}) and (\ref{1.7}), we obtain the approximation operator of the Riesz tempered fractional derivative

\begin{equation}\label{2.24}
\begin{split}
 \nabla_x^{\alpha,\lambda} G(x,\rho, t)\Big|_{(x=x_i,t=t_n)}
&=-\kappa_{\alpha}\left( _{a}\nabla_x^{\alpha,\lambda}+ {_x}\nabla_{b}^{\alpha,\lambda} \right)G(x,\rho,t)\Big|_{(x=x_i,t=t_n)}\\
%&=-\frac{\kappa_{\alpha}}{h^\alpha} \left[\sum_{j=0}^{i+1}e^{-(j-1)\lambda h}\omega_j^{\alpha}G(x_{i-j+1},\rho,t_n)
%  +\sum_{j=0}^{M-i+1}e^{-(j-1)\lambda h}\omega_j^{\alpha}G(x_{i+j-1},\rho,t_n)  \right] +\mathcal{O}(h)^2\\
%&=-\frac{\kappa_{\alpha}}{h^\alpha} \left[\sum_{j=0}^{i+1}e^{-(i-j)\lambda h}\omega_{i-j+1}^{\alpha}G(x_{j},\rho,t_n)
%  +\sum_{j=i-1}^{M}e^{-(j-i)\lambda h}\omega_{j-i+1}^{\alpha}G(x_{j},\rho,t_n)  \right] +\mathcal{O}(h)^2\\
&=-\frac{\kappa_{\alpha}}{h^\alpha} \sum_{j=0}^{M}\omega_{i,j}^{\alpha}G(x_{j},\rho,t_n)   +\mathcal{O}(h^2).
\end{split}
\end{equation}
Here
\begin{equation}\label{2.25}
\omega_{i,j}^{\alpha}=\left\{ \begin{array}
 {l@{\quad } l}
  e^{-(i-j)\lambda h}\omega_{i-j+1}^{\alpha},&j < i-1,\\
 e^{\lambda h} \omega_{0}^{\alpha}+e^{-\lambda h}\omega_{2}^{\alpha} ,&j=i-1,\\
 2\omega_{1}^{\alpha},&j=i,\\
 e^{\lambda h}\omega_{0}^{\alpha}+e^{-\lambda h}\omega_{2}^{\alpha} ,&j=i+1,\\
 e^{-(j-i)\lambda h}\omega_{j-i+1}^{\alpha} ,&j>i+1\\
 \end{array}
 \right.
\end{equation}
with $i=1,\ldots,M-1$, together with the Dirichlet boundary conditions that define $G(x_{0},\rho,t_n)$ and $G(x_{M},\rho,t_n)$ appropriately; and $\omega_j^{\alpha}$  is given in (\ref{2.11}) with the parameters $r_3$ satisfying (\ref{2.16}) and $r_1$ and $r_2$ specified in (\ref{2.6}).

From (\ref{2.20}), (\ref{2.21}), and (\ref{2.24}),   we can write (\ref{1.1}) as
\begin{equation}\label{2.26}
\begin{array}{ll}
&\displaystyle \frac{1}{\tau^\gamma}\sum_{k=0}^{n}{d}_{i,k}^{v,\gamma}G(x_i,\rho,t_{n-k})
  -\frac{1}{\tau^\gamma}\sum_{k=0}^{n}d_{i,k}^{v,\gamma}e^{J\rho U_i (n-k)\tau}G(x_i,\rho,0)\\
&\displaystyle =-\frac{K\kappa_\alpha}{h^\alpha}
\sum_{j=0}^{M}\omega_{i,j}^{\alpha}G(x_{j},\rho,t_n)+ r_i^{n}
\end{array}
\end{equation}
with the local truncation error
\begin{equation}\label{2.27}
  |r_i^n| \leq C_G(\tau^\nu+h^2),~~\nu=1,2,3,4,
\end{equation}
where $C_G$ is a constant independent of $\tau$ and $h$.

Multiplying (\ref{2.26}) by $\tau^\gamma$ leads to
\begin{equation}\label{2.28}
\begin{array}{l}
\displaystyle\sum_{k=0}^{n}{d}_{i,k}^{\nu,\gamma}G(x_i,\rho,t_{n-k})
-\sum_{k=0}^{n}d_{i,k}^{\nu,\gamma}e^{J\rho U_i (n-k)\tau}G(x_i,\rho,0)\\
\displaystyle=\kappa\sum_{j=0}^{M}\omega_{i,j}^{\alpha}G(x_{j},\rho,t_{n})+ R_i^{n},
\end{array}
\end{equation}
where
\begin{equation}\label{2.29}
\kappa=-\frac{K\kappa_\alpha\tau^\gamma}{h^\alpha}>0,~~\alpha\in(1,2),
\end{equation}
and
\begin{equation}\label{2.30}
|R_i^n|=|\tau^{\gamma}r_i^n| \leq C_G \tau^{\gamma}(\tau^\nu+h^2),~~\nu=1,2,3,4.
\end{equation}
From (\ref{2.28}), the resulting discretization of (\ref{1.1}) can be rewritten as
\begin{equation}\label{2.31}
\begin{split}
& d_{i,0}^{\nu,\gamma}G_{i,\rho}^n-\kappa\sum_{j=0}^{M}\omega_{i,j}^{\alpha} G_{j,\rho}^n =\sum_{k=0}^{n-1}d_{i,k}^{\nu,\gamma}e^{J\rho U_i  (n-k)\tau}G_{i,\rho}^{0}-\sum_{k=1}^{n-1}d_{i,k}^{\nu,\gamma}G_{i,\rho}^{n-k},~~n\geq 1.
\end{split}
\end{equation}
It is worthwhile noting that the second term
on the right hand side of (\ref{2.31})  automatically vanishes when $n =1$.

For convenience of implementation, we use the matrix form of the grid function
 \begin{equation*}
 \widetilde{G}^{n}=[G_{1,\rho}^n,G_{2,\rho}^n,\ldots,G_{M-1,\rho}^n]^{\rm T};
  \end{equation*}
and the finite difference scheme (\ref{2.31}) can be recast as
 \begin{equation}\label{2.32}
\begin{split}
&\left(d_{i,0}^{\nu,\gamma} I - \kappa H\right) \widetilde{G}^{n}=\sum_{k=0}^{n-1}g_k^\gamma e^{J\rho U_i  n\tau}\widetilde{G}^{0}
-\sum_{k=1}^{n-1}g_k^\gamma e^{J\rho U_i k \tau}\widetilde{G}^{n-k},
\end{split}
\end{equation}
where the matrix $H$ is defined by (\ref{2.18}) and $I$ is the identity matrix. It should be noticed that $H$ has the Toeplitz structure and then the computation cost is $\mathcal{O}(M\log M)$ in numerically solving the equation by the iteration methods, e.g., multigrid method.

In particular, when $\nu=1$, from (\ref{2.31}) and (\ref{2.23}), we obtain the scheme
\begin{equation}\label{2.33}
\begin{split}
& G_{i,\rho}^n-\kappa\sum_{j=0}^{M}\omega_{i,j}^{\alpha} G_{j,\rho}^n
 =\sum_{k=0}^{n-1}g_k^\gamma e^{J\rho U_i  n\tau}G_{i,\rho}^{0}-\sum_{k=1}^{n-1}g_k^\gamma e^{J\rho U_i k \tau}G_{i,\rho}^{n-k},~~n\geq 1;
\end{split}
\end{equation}
and it can be rewritten as
\begin{equation}\label{2.34}
\begin{split}
&(1-\kappa \omega_{i,i}^{\alpha} )G_{i,\rho}^1-\kappa\!\!\sum_{j=0,j\neq i}^{M}\omega_{i,j}^{\alpha} G_{j,\rho}^1= e^{J\rho U_i  \tau}G_{i,\rho}^{0},~~n=1;\\
&(1-\kappa \omega_{i,i}^{\alpha} ) G_{i,\rho}^n-\kappa\!\!\sum_{j=0,j\neq i}^{M}\omega_{i,j}^{\alpha} G_{j,\rho}^n\\
& =\sum_{k=0}^{n-1}g_k^\gamma e^{J\rho U_i  n\tau}G_{i,\rho}^{0}
   -\sum_{k=1}^{n-1}g_k^\gamma e^{J\rho U_i k \tau}G_{i,\rho}^{n-k},~~n> 1.
\end{split}
\end{equation}

\subsection{Detailed proof of the numerical stability and convergence}
In this subsection, we theoretically prove that the provided first order time discretization scheme (\ref{2.34}) is unconditionally stable and convergent. Denote $v^n=\{v_i^n\,|\, 0 \leq i \leq M, n \geq 0,\,v_0=v_{M}=0\}$,
which is a grid function.  And we define the pointwise maximum norm as
$||v^n||_\infty=\max\limits_{0\leq i \leq M}|v_i^n|$
and the discrete $L^2$ norm as
$||v^n||=\sqrt{h\sum\limits_{i=1}^{M-1}v_i^2}.$
\begin{lemma}\cite{Chen:09,Deng:14}\label{lemma2.12}
The coefficients $g_k^\gamma$ defined in (\ref{2.23}) with $\gamma \in (0,1)$ satisfy
\begin{equation*}
\begin{split}
 & g^\gamma_0=1; ~~~~ g^\gamma_k<0,~~ (k\geq 1); ~~~~\sum_{k=0}^{n-1}g^\gamma_k>0;  ~~~~\sum_{k=0}^{\infty}g^\gamma_k=0;\\
 {\rm and}&~~~~ \frac{1}{n^\gamma \Gamma(1-\gamma)}< \sum_{k=0}^{n-1}g^\gamma_k=-\sum_{k=n}^{\infty}g^\gamma_k\leq \frac{1}{n^\gamma} ~~{\rm for}~~n\geq 1.
  \end{split}
\end{equation*}
\end{lemma}

\begin{lemma}\label{lemma2.13}
The coefficients $\omega_{i,j}^{\alpha}$ defined in (\ref{2.25}) with $\alpha\in (1,2)$ satisfy
\begin{equation*}
  \begin{split}
&(1) ~~ \omega_{i,i}^{\alpha}<0,\,\,\,\,\,\, \omega_{i,j}^{\alpha} > 0\,\,(j\neq i);\\
&(2) ~~ \sum\limits_{j=0}^{M}\omega_{i,j}^{\alpha}< 0 ~~ \mbox {and} ~~  -\omega_{i,i}^{\alpha}>\!\!\!\!\!
         \sum\limits_{j=0,j\neq i}^{M}\!\!\!\!\omega_{i,j}^{\alpha}.
  \end{split}
\end{equation*}
\end{lemma}
\begin{proof}
From Lemmas \ref{lemma2.9} and \ref{lemma2.10}, the desired results are easily obtained.
\end{proof}

\begin{lemma}\cite{Deng:14}\label{lemma2.14}
Let $R \geq 0$; $\varepsilon^k \geq 0$,~$k=0,1,\ldots,N$ and satisfy
\begin{equation*}
  \varepsilon^n\leq  - \sum_{k=1}^{n-1} g_k^\gamma  \varepsilon^{n-k}+R,~~n\geq 1.
\end{equation*}
Then we have the following estimates:

(a) when $0<\gamma<1$,
$
\varepsilon^n \leq   \left(\sum\limits_{k=0}^{n-1}g_k^\gamma\right)^{-1}R \leq  n^\gamma \Gamma(1-\gamma)R;
$

(b) when $\gamma \rightarrow1$,
$
\varepsilon^n \leq     nR.
$
\end{lemma}

\begin{theorem}\label{theorem2.15}
The difference scheme (\ref{2.33})  is unconditionally stable.
\end{theorem}
\begin{proof}
Let ${\widetilde G_{i,\rho}}^n$ be the approximate solution of $G_{i,\rho}^n$,
which is the exact solution of the scheme (\ref{2.33}). Taking $\varepsilon_i^n={\widetilde G_{i,\rho}}^n- G_{i,\rho}^n$,
then from (\ref{2.34}) we get the following perturbation equation
\begin{equation}\label{2.35}
\begin{split}
&(1-\kappa \omega_{i,i}^{\alpha} )\varepsilon_i^1  -    \kappa\!\!\sum_{j=0,j\neq i}^{M}\omega_{i,j}^{\alpha} \varepsilon_j^1
   = e^{J\rho U_i  \tau}\varepsilon_i^{0},~~n=1,\\
&(1-\kappa \omega_{i,i}^{\alpha} ) \varepsilon_i^n-\kappa\!\!\!\!\sum_{j=0,j\neq i}^{M}\!\!\!\!\omega_{i,j}^{\alpha} \varepsilon_j^n
 =\sum_{k=0}^{n-1}g_k^\gamma e^{J\rho U_i  n\tau}\varepsilon_i^{0}-\sum_{k=1}^{n-1}g_k^\gamma e^{J\rho U_i k \tau}\varepsilon_i^{n-k},~~n> 1.
\end{split}
\end{equation}

Denote $\varepsilon^n=[\varepsilon_0^n,\varepsilon_1^n,\ldots, \varepsilon_M^n]$
and $||\varepsilon^n||_{\infty}=\max \limits _{0\leq i \leq M}|\varepsilon_i^n|$.
Next we prove that $||\varepsilon^n||_{\infty} \leq ||\varepsilon^0||_{\infty}$ by the  mathematical induction.

For  $n=1$, suppose $|\varepsilon_{i_0}^1|=||\varepsilon^1||_{\infty}=\max \limits _{0\leq i \leq M}|\varepsilon_i^1|$.
From (\ref{2.35}), we obtain
\begin{equation}\label{2.36}
\begin{split}
(1-\kappa \omega_{i_0,i_0}^{\alpha} )\varepsilon_{i_0}^1-\kappa\!\!\!\!\sum_{j=0,j\neq {i_0}}^{M}\!\!\!\!\omega_{i_0,j}^{\alpha} \varepsilon_j^1
= e^{J\rho U_{i_0}  \tau}\varepsilon_{i_0}^{0},~~n=1,~~J=\sqrt{-1}.
\end{split}
\end{equation}
Then
\begin{equation*}
\begin{split}
||\varepsilon^1||_{\infty}
&=|\varepsilon_{i_0}^1| \leq |\varepsilon_{i_0}^1|-\kappa\sum_{j=0}^{M}\omega_{i_0,j}^{\alpha} |\varepsilon_{i_0}^1|
        =(1-\kappa\omega_{i_0,i_0}^\alpha )|\varepsilon_{i_0}^1|-\kappa\!\!\!\sum_{j=0,j\neq {i_0}}^{M}\!\!\!\omega_{i_0,j}^\alpha |\varepsilon_{i_0}^1|\\
&\leq (1-\kappa\omega_{i_0,i_0}^\alpha )|\varepsilon_{i_0}^1|-\kappa\!\!\!\sum_{j=0,j\neq {i_0}}^{M}\!\!\!\omega_{i_0,j}^\alpha |\varepsilon_{j}^1|
\leq \left|(1-\kappa\omega_{i_0,i_0}^\alpha )\varepsilon_{i_0}^1-\kappa\!\!\!\sum_{j=0,j\neq {i_0}}^{M}\!\!\!\omega_{i_0,j}^\alpha \varepsilon_{j}^1 \right|\\
&=|e^{J\rho U_{i_0}  \tau}\varepsilon_{i_0}^{0}|= |\varepsilon_{i_0}^{0}| \leq ||\varepsilon^0||_{\infty}.
\end{split}
\end{equation*}
Supposing $|\varepsilon_{i_0}^n|=||\varepsilon^n||_{\infty}=\max \limits _{0\leq i \leq M}|\varepsilon_i^n|$,
from Lemma \ref{lemma2.13} and (\ref{2.35}), we obtain
\begin{equation}\label{2.37}
\begin{split}
||\varepsilon^n||_{\infty}
&\leq \left|(1-\kappa\omega_{i_0,i_0}^\alpha )\varepsilon_{i_0}^n-\kappa\!\!\!\sum_{j=0,j\neq {i_0}}^{M}\!\!\!\omega_{i_0,j}^\alpha \varepsilon_{j}^n \right|\\
&=\left |\sum_{k=0}^{n-1}g_k^\gamma e^{J\rho U_{i_0}  n\tau}\varepsilon_{i_0}^{0}
      -\sum_{k=1}^{n-1}g_k^\gamma e^{J\rho U_{i_0} k \tau}\varepsilon_{i_0}^{n-k}\right|\\
&\leq \left |\sum_{k=0}^{n-1}g_k^\gamma e^{J\rho U_{i_0}  n\tau}\varepsilon_{i_0}^{0}\right|
      +\left |\sum_{k=1}^{n-1}g_k^\gamma e^{J\rho U_{i_0} k \tau}\varepsilon_{i_0}^{n-k}\right|.
\end{split}
\end{equation}
Using Lemma \ref{lemma2.12}, we get
\begin{equation}\label{2.38}
\begin{split}
\left |\sum_{k=0}^{n-1}g_k^\gamma e^{J\rho U_{i_0}  n\tau}\varepsilon_{i_0}^{0}\right|
%& =\left |\varepsilon_{i_0}^{0}\right| \left |\sum_{k=0}^{n-1}g_k^\gamma e^{J\rho U_{i_0}  n\tau}\right| \\
& =\left |\varepsilon_{i_0}^{0}\right|\cdot \left |e^{J\rho U_{i_0}  n\tau}\right| \cdot \left |\sum_{k=0}^{n-1}g_k^\gamma \right|
=\sum_{k=0}^{n-1}g_k^\gamma ||\varepsilon^0||_{\infty},
\end{split}
\end{equation}
and
\begin{equation}\label{2.39}
\begin{split}
\left |\sum_{k=1}^{n-1}g_k^\gamma e^{J\rho U_{i_0} k \tau}\varepsilon_{i_0}^{n-k}\right|
%&\leq \sum_{k=1}^{n-1} \left | g_k^\gamma e^{J\rho U_{i_0} k \tau}\right| \left |\varepsilon_{i_0}^{n-k}\right|\\
&\leq \sum_{k=1}^{n-1} \left | g_k^\gamma\right|\cdot ||\varepsilon^{n-k}||_{\infty}
=-\sum_{k=1}^{n-1}  g_k^\gamma ||\varepsilon^{n-k}||_{\infty}.
\end{split}
\end{equation}
Thus, according to (\ref{2.37})-(\ref{2.39}), there exists
\begin{equation}\label{2.40}
\begin{split}
||\varepsilon^n||_{\infty}
\leq \sum_{k=0}^{n-1}  g^\gamma_k ||\varepsilon^0||_{\infty}-\sum_{k=1}^{n-1}g^\gamma_k||\varepsilon^{n-k}||_{\infty}.
\end{split}
\end{equation}
Next we prove the following inequality holds by the mathematical induction
$$||\varepsilon^n||_{\infty} \leq ||\varepsilon^0||_{\infty}, ~~\forall n \geq 1.$$
In fact, for $n=1$, Eq. (\ref{2.40}) holds obviously. Suppose that
$$||\varepsilon^s||_{\infty} \leq ||\varepsilon^0||_{\infty}, ~~s=1,2,\ldots,n-1.$$
Then from (\ref{2.40}), it implies that
$$
||\varepsilon^n||_{\infty}
\leq \sum_{k=0}^{n-1}  g^\gamma_k ||\varepsilon^0||_{\infty}
 -\sum_{k=1}^{n-1}g^\gamma_k||\varepsilon^{n-k}||_{\infty}
\leq \sum_{k=0}^{n-1}  g^\gamma_k ||\varepsilon^0||_{\infty}
 -\sum_{k=1}^{n-1}g^\gamma_k||\varepsilon^{0}||_{\infty}=||\varepsilon^{0}||_{\infty}.
$$
Thus, the proof is completed.
\end{proof}

\begin{theorem}\label{theorem2.16}
Let $G(x_i,\rho,t_n)$  be the exact solution of (\ref{2.28}), and $G_{i,\rho}^n$ the solution of the  finite difference
scheme (\ref{2.33}).  Then the error estimates are
$$
||G(x_i,\rho,t_n)-G_{i,\rho}^n||_\infty \leq  C_G\Gamma(1-\gamma) T^\gamma (\tau+h^2),~~{\rm for}~~0<\gamma <1;
$$
and
$$
||G(x_i,\rho,t_n)-G_{i,\rho}^n||_\infty \leq  C_GT  \tau^{\gamma-1}(\tau+h^2),~~{\rm for}~~\gamma \rightarrow 1,
$$
where $C_G$ is defined by  (\ref{2.27}), $i=0,1,\ldots,M;~ n=1,2,\ldots,N$.
\end{theorem}

\begin{proof}
Let $G(x_i,\rho,t_n)$ be the exact solution of (\ref{2.28}) at the mesh point $(x_i,t_n)$,
and $G_{i,\rho}^n$ the  solution of the  finite difference scheme (\ref{2.33}).
Denote $e_i^n=G(x_i,\rho,t_n)-G_{i,\rho}^n$ and  $e^n=[e_0^n,e_1^n,\ldots, e_{M}^n]^T$.
Subtracting (\ref{2.28}) from (\ref{2.33}) and using $e_i^0=0$, we obtain
\begin{equation}\label{2.41}
\begin{split}
&(1-\kappa \omega_{i,i}^{\alpha} )e_i^1  -    \kappa\!\!\sum_{j=0,j\neq i}^{M}\omega_{i,j}^{\alpha} e_j^1
   = R_i^1,~~n=1,\\
&(1-\kappa \omega_{i,i}^{\alpha} ) e_i^n-\kappa\!\!\sum_{j=0,j\neq i}^{M}\omega_{i,j}^{\alpha} e_j^n
 =-\sum_{k=1}^{n-1}g_k^\gamma e^{J\rho U_i k \tau}e_i^{n-k}+R_i^n,~~n> 1,
\end{split}
\end{equation}
where  $R_i^n $ is defined by (\ref{2.30}) with $\nu=1$.

Denoting that  $||e^n||_{\infty}=\max \limits _{0\leq i \leq M}|e_i^n|$
and $R_{\max}=\max \limits _{0\leq i \leq M,0\leq n \leq N}|R_i^n|$,  the desired result can be proved by using the mathematical induction.

For  $n=1$, supposing  $|e_{i_0}^1|=||e^1||_{\infty}=\max \limits _{0\leq i \leq M}|e_i^1|$
and using  (\ref{2.41}), we get
\begin{equation*}
\begin{split}
(1-\kappa \omega_{i_0,i_0}^{\alpha} )e_{i_0}^1-\kappa\!\!\sum_{j=0,j\neq {i_0}}^{M}\omega_{i_0,j}^{\alpha} e_j^1= R_{i_0}^1.
\end{split}
\end{equation*}
According to Lemma \ref{lemma2.13} and the above equation, we obtain
\begin{equation*}
\begin{split}
||e^1||_{\infty}
=|e_{i_0}^1|
% \leq |e_{i_0}^1|-\kappa\sum_{j=0}^{M}\omega_{i_0,j}^{\alpha} |e_{i_0}^1|
%        =(1-\kappa\omega_{i_0,i_0}^\alpha )|e_{i_0}^1|-\kappa\!\!\!\sum_{j=0,j\neq {i_0}}^{M}\!\!\!\omega_{i_0,j}^\alpha |e_{i_0}^1|\\
&\leq \left|(1-\kappa\omega_{i_0,i_0}^\alpha )e_{i_0}^1-\kappa\!\!\!\sum_{j=0,j\neq {i_0}}^{M}\!\!\!\omega_{i_0,j}^\alpha e_{j}^1 \right|
=|R_{i_0}^1|  \leq R_{\max}.
\end{split}
\end{equation*}
Supposing $|e_{i_0}^n|=||e^n||_{\infty}=\max \limits _{0\leq i \leq M}|e_i^n|$, and  from    (\ref{2.39}), (\ref{2.41}), Lemma \ref{lemma2.13}, there exists
\begin{equation*}
\begin{split}
||e^n||_{\infty}
&\leq \left|(1-\kappa\omega_{i_0,i_0}^\alpha )e_{i_0}^n-\kappa\!\!\!\sum_{j=0,j\neq {i_0}}^{M}\!\!\!\omega_{i_0,j}^\alpha e_{j}^n \right|\\
&=\left | -\sum_{k=1}^{n-1}g_k^\gamma e^{J\rho U_{i_0} k \tau}e_{i_0}^{n-k}+R_{i_0}^n\right|
       \leq \left |\sum_{k=1}^{n-1}g_k^\gamma e^{J\rho U_{i_0} k \tau}e_{i_0}^{n-k}\right|+R_{\max}\\
&\leq -\sum_{k=1}^{n-1}  g_k^\gamma ||e^{n-k}||_{\infty}+R_{\max},
\end{split}
\end{equation*}
i.e.,
$||e^n||_{\infty} \leq -\sum\limits_{k=1}^{n-1}g^\gamma_k||e^{n-k}||_{\infty} +R_{\max}.$

Hence, from Lemma \ref{lemma2.14},  we have the following estimates

(a) when $0<\gamma<1$,
\begin{equation*}
 ||e^n||_{\infty} \leq   \left(\sum_{k=0}^{n-1}g_k^\gamma\right)^{-1}R_{\max} \leq  n^\gamma \Gamma(1-\gamma)R_{\max}
 \leq  C_G\Gamma(1-\gamma) T^\gamma  (\tau+h^2);
\end{equation*}

(b) when $\gamma \rightarrow1$,
$||e^n||_{\infty} \leq     nR_{\max}\leq  C_GT  \tau^{\gamma-1}(\tau+h^2).$
\end{proof}

Besides the $L_\infty$ norm, the unconditional stability and convergence can also be obtained in the discrete $L^2$ norm.
In the following theorem, we present the convergence result in $L^2$ norm; because of the similarity,
we omit the proof of unconditional stability in $L^2$ norm.

\begin{theorem}\label{theorem2.17}
Let $G(x_i,\rho,t_n)$  be the exact solution of (\ref{2.28}), and $G_{i,\rho}^n$ the solution of the  finite difference
scheme (\ref{2.33}).  Then the error estimates are
$$
||G(x_i,\rho,t_n)-G_{i,\rho}^n|| \leq (b-a)^{\frac{1}{2}} C_G\Gamma(1-\gamma) T^\alpha (\tau+h^2),~~{\rm for}~~0<\gamma <1;
$$
and
$$
||G(x_i,\rho,t_n)-G_{i,\rho}^n|| \leq (b-a)^{\frac{1}{2}} C_GT  \tau^{\gamma-1}(\tau+h^2),~~{\rm for}~~\gamma \rightarrow 1,
$$
where $C_G$ is defined by  (\ref{2.27}), $i=0,1,\ldots,M;~ n=1,2,\ldots,N$.
\end{theorem}

\begin{proof}
Let $G(x_i,\rho,t_n)$ be the exact solution of (\ref{2.28}) at the mesh point $(x_i,t_n)$,
and  $G_{i,\rho}^n$ the  solution of the  finite difference scheme (\ref{2.33}).
Denote $e_i^n=G(x_i,\rho,t_n)-G_{i,\rho}^n$ and  $e^n=[e_1^n,e_2^n,\ldots, e_{M-1}^n]^T$.
Subtracting (\ref{2.28}) from (\ref{2.33}) and using $e_i^0=0$, we obtain
$$
\left(I - \kappa H\right) e^{n}=-\sum_{k=1}^{n-1}g_k^\gamma e^{J\rho U_i k \tau}e^{n-k}+R^n,
$$
where $H$ is defined by (\ref{2.18}), and $R^n=[R_1^n,R_2^n,\ldots, R_{M-1}^n]^T$.

Performing the inner product in both sides of the above equation by $e^n$ leads to
\begin{equation*}
\left(\left(I - \kappa H\right) e^{n},e^n\right)=\left(-\sum_{k=1}^{n-1}g_k^\gamma e^{J\rho U_i k \tau}e^{n-k},e^n\right)+(R^n,e^n).
\end{equation*}
From Definition \ref{definition2.5} and Theorem \ref{theorem2.11}, we obtain $\left( - \kappa H e^{n},e^n\right) \geq 0$. Then
\begin{equation*}
||e^n||^2\leq \left(-\sum_{k=1}^{n-1}g_k^\gamma e^{J\rho U_i k \tau}e^{n-k},e^n\right)+(R^n,e^n);
\end{equation*}
it follows that
\begin{equation}\label{2.42}
||e^n|| \leq \Big|\Big|-\sum_{k=1}^{n-1}g_k^\gamma e^{J\rho U_i k \tau}e^{n-k}\Big|\Big|+||R^n||.
\end{equation}
Then
\begin{equation*}
\begin{split}
  ||e^n|| \leq -\sum_{k=1}^{n-1}g^\gamma_k||e^{n-k}|| +R_{\max},
\end{split}
\end{equation*}
where $R_{\max} \leq (b-a)^{\frac{1}{2}} C_G \tau^\gamma(\tau+h^2)$.

Hence, from Lemma \ref{lemma2.14},  we have the estimates

(a) when $0<\gamma<1$,
\begin{equation*}
 ||e^n|| \leq   \left(\sum_{k=0}^{n-1}g_k^\gamma\right)^{-1}R_{\max} \leq  n^\gamma \Gamma(1-\gamma)R_{\max}
 \leq (b-a)^{\frac{1}{2}} C_G\Gamma(1-\gamma) T^\gamma  (\tau+h^2);
\end{equation*}

(b) when $\gamma \rightarrow1$,
$||e^n|| \leq     nR_{\max}\leq (b-a)^{\frac{1}{2}} C_GT  \tau^{\gamma-1}(\tau+h^2).$
\end{proof}

\subsection{Numerical results}
%We numerically verify the above theoretical results including convergent
%orders and numerical stability.  And the $ l_\infty$ norm is used to measure the numerical errors.
We employ the V-cycle Multigrid method (MGM)   \cite{Chen:14,Pang:12} to solve
 (\ref{1.1}) with the algorithms given in \S 2.2; and the parameters of MGM, e.g., `Iter', `CPU', etc, are  the same as the ones of  \cite{Chen:14}. For the convenience to the readers, the MGM's pseudo codes are added in the Appendix.
All the numerical experiments are programmed in Python, and the computations are carried out on a PC with the configuration:
Intel(R) Core(TM) i5-3470 3.20 GHZ and 8 GB RAM and a 64 bit Windows 7 operating system.
Without loss of generality, we add a force term $f(x,\rho,t)$ on the right side of (\ref{1.1}).

%%%%%%%%%%%%%%%%%%%%%%%%%%%%%%%%%%%%%  Example 5.2  %%%%%%%%%%%%

%%%%%%%%%%%%%%%%%%%%%%%%%%  Example 5.3  %%%%%%%%%%%%%%%%%%%%%
\begin{example}\label{example12.1}\end{example}
Consider (\ref{1.1}) on a finite domain with $x \in (a,b)$ ($a=0$, $b=1$),  $0<t \leq 1$,  and the coefficient  $K=1$, $U(x)=x$, $\rho=1$,   $J=\sqrt{-1}$; the forcing function
\begin{equation*}
\begin{split}
f(x,\rho,t)=&\frac{\Gamma(4+\gamma)}{\Gamma(4)}e^{J\rho x t}t^3\sin(x^2)\sin((1-x)^2)\\
       &  +\frac{1}{2\cos(\alpha \pi/2)} (t^{3+\gamma}+1)e^{-\lambda x}{ _{a}}D_x^{\alpha}[e^{(\lambda+J\rho t) x}\sin(x^2)\sin((1-x)^2)]\\
       &  +\frac{1}{2\cos(\alpha \pi/2)} (t^{3+\gamma}+1)e^{\lambda x}{ _{x}}D_b^{\alpha}[e^{(-\lambda+J\rho t) x}\sin(x^2)\sin((1-x)^2)]\\
       &  -\frac{\lambda^\alpha}{\cos(\alpha \pi/2)} e^{J\rho xt}(t^{3+\gamma}+1)\sin(x^2)\sin((1-x)^2),\\
\end{split}
\end{equation*}
where the left and right fractional derivatives of the given functions are calculated by Algorithm \ref{AFD} presented in the Appendixes.
The initial condition $G(x,\rho,0)=\sin(x^2)\sin((1-x)^2) $, and the boundary
conditions $G(0,\rho,t)=G(1,\rho,t)=0$. Then (\ref{1.1}) has the exact
solution $G(x,\rho,t)=e^{J\rho xt}(t^{3+\gamma}+1)\sin(x^2)\sin((1-x)^2). $

\begin{table}[h]\fontsize{9.5pt}{12pt}\selectfont%生成浮动表格
  \begin{center}%\def\tabcolsep{28.5pt}%表格居中
  \caption{MGM   to solve the scheme (\ref{2.31}) with $\nu=2$ at $T=1$ and $N=M$,
 where  $\lambda=0.2$,  $U(x)=x$, $\rho=1$,  $r_3=0$. }\vspace{5pt}%标题，离表格一定的距离
 { \begin{tabular*}{\linewidth}{@{\extracolsep{\fill}}*{9}{c}}                                    \hline  %画顶端的横线
$M$ &  $\alpha=1.3,\gamma=0.8$  & Rate & Iter     & CPU  & $\alpha=1.8,\gamma=0.3$&   Rate   &  Iter   &  CPU    \\\hline
    $2^4$   &   1.3494e-003   &           &  7.0     &0.29 s       &  1.6256e-003    &          &  8.0    & 0.26 s\\\hline %画底端的横线
    $2^5$   &   3.3193e-004    & 2.02    &  6.0     & 0.83 s       &  3.9709e-004  &  2.03  &  8.0    & 0.92 s \\\hline %画底端的横线
    $2^6$   &   8.1137e-005  & 2.03    &  6.0     & 2.57 s       &  9.6906e-005   &  2.03  &  8.0    & 2.93 s \\\hline % 画底端的横线
    $2^7$   &   2.0198e-005  &2.01    &  6.0     & 9.03 s       &  2.3633e-005   &  2.04  &  7.0    & 10.01 s \\\hline % 画底端的横线
    \end{tabular*}}\label{table:1}%\vspace{-15pt}
  \end{center}
\end{table}

\begin{table}[h]\fontsize{9.5pt}{12pt}\selectfont%生成浮动表格
  \begin{center}%\def\tabcolsep{28.5pt}%表格居中
  \caption{MGM  to solve the scheme (\ref{2.31}) with $\nu=1$ at $T=1$ and $N=M$,
 where  $\lambda=0.7$,  $U(x)=x$, $\rho=1$,   $r_3= \frac{(\alpha-1)(2-\alpha)(\alpha+3)}{4(\alpha+1)(\alpha+2)}$. }\vspace{5pt}%标题，离表格一定的距离
 { \begin{tabular*}{\linewidth}{@{\extracolsep{\fill}}*{9}{c}}                                    \hline  %画顶端的横线
$M$ &  $\alpha=1.3,\gamma=0.8$  & Rate & Iter     & CPU  & $\alpha=1.8,\gamma=0.3 $&   Rate   &  Iter   &  CPU    \\\hline
    $2^4$   &   2.4702e-003   &           &  8.0     & 0.23 s       & 1.7526e-003   &          &  9.0    & 0.28 s\\\hline %画底端的横线
    $2^5$   &   1.2761e-003  & 0.95    &  7.0     & 0.72 s       &  5.1049e-004   &  1.79  &  9.0    & 0.82 s \\\hline % 画底端的横线
    $2^6$   &   6.4040e-004  &0.99   &  6.0     & 2.41 s       &   1.6173e-004   &  1.66 &  9.0    & 2.85 s \\\hline %画底端的横线
    $2^7$   &   3.2027e-004  & 1.00    &  6.0     & 8.68 s       &  5.7602e-005  &  1.49  &  10.0   &10.18 s \\\hline % 画底端的横线
    \end{tabular*}}\label{table:2}%\vspace{-15pt}
  \end{center}
\end{table}

Table \ref{table:1}  shows that the schemes (\ref{2.31}) with $\nu=2$ have the global truncation errors $\mathcal{O} (\tau^2+h^2)$  at time $T=1$, and numerically confirms that the computational cost is of $\mathcal{O}(M \mbox{log} M)$ operations. Similarly, Table \ref{table:2}  shows that the algorithms (\ref{2.31}) with $\nu=1$ have the global truncation errors $\mathcal{O} (\tau+h^2)$  at time $T=1$, and the computational cost also is of $\mathcal{O}(M \mbox{log} M)$ operations.

%%%%%%%%%%%%%%%%%%%%%%%%%%%%%%%%%%%%%
%%%%%%%%%%%%%%%%%%%%%%%%%%%%%%%%%%%%%

\section{High order schemes for (\ref{1.1}) with nonhomogeneous boundary and/or initial conditions}
Generally, when the high order finite difference discretizations are used to solve the fractional differential equations, the potential analytical solution and its several derivatives must be zero at the boundaries and/or initial time  for keeping the high accuracy. These requirements greatly limit the practical applications of the high order schemes to obtain high accuracy. This section provides the techniques to overcome the challenges, i.e., modifies the current high order schemes such that they can keep high accuracy without the above requirements.
\subsection{Modifications to the high order discretizations}
Let $\sigma$ be a constant or a function without related to $x$, say $\sigma(y)$, and denote
\begin{equation}\label{3.1}
\begin{split}
{_a^1}D_{x}^{q,\sigma}G(x) =\left( \frac{\partial}{\partial x} +\sigma \right)^q\!\!G(x)~~{\rm and }~~
{_x^1}D_{b}^{q,\sigma}G(x) =(-1)^q\left( \frac{\partial}{\partial x} -\sigma \right)^q\!\!G(x),
\end{split}
\end{equation}
where $q$ is a positive integer.

Then, from the  $\nu$-th order difference formula  \cite[p.\,83] {Gustafsson:08} and (\ref{3.1}), we obtain
\begin{equation}\label{3.2}
\begin{split}
&{_a^1}D_{x}^{q,\sigma}G(x)|_{x=x_0} =\frac{1}{h^q}\sum_{p=0}^{q+\nu-1}b_p^{q,\sigma h}G(x_p)+\mathcal{O}(h^\nu)
\end{split}
\end{equation}
and
\begin{equation}\label{3.3}
\begin{split}
&{_x^1}D_{b}^{q,\sigma}G(x)|_{x=x_M} =\frac{1}{h^q}\sum_{p=M-q-\nu+1}^{M}\overline{b}_p^{q,\sigma  h}G(x_p)+\mathcal{O}(h^\nu),
\end{split}
\end{equation}
where $x_0=a$, $x_M=b$, and $x_p=a+ph$. The coefficients $b_p^{q,\sigma  h}$ and $\overline{b}_p^{q,\sigma  h}$ are, respectively, given in Tables \ref{Table3.1} and \ref{Table3.2}.

\begin{table}[h]\fontsize{9.5pt}{12pt}\selectfont%生成浮动表格
  \begin{center}%\def\tabcolsep{28.5pt}%表格居中
  \caption{Coefficients $b_p^{q,\sigma  h}$  for approximations (\ref{3.2}) }\vspace{5pt}%标题，离表格一定的距离
 { \begin{tabular*}{\linewidth}{@{\extracolsep{\fill}}*{9}{c}}                                    \hline  %画顶端的横线
q &  $\nu$ & $x_0$             & $x_1$     & $x_2$  & $x_3 $   &$x_4 $  &  $x_5 $  \\\hline
                                &   1    &$-1+\sigma h$      &  1        &        &          &          \\
                                &   2    &$-\frac{3}{2}+\sigma h$    &  2        & $-\frac{1}{2}$ &          &          \\
1 &   3    &$-\frac{11}{6}+\sigma h$   &  3        & $-\frac{3}{2}$ & $\frac{1}{3}$    &          \\
                                &   4    &$-\frac{25}{12}+\sigma h$  &  4        & -3     &  $\frac{4}{3}$   & $-\frac{1}{4}$  \\\hline %画底端的横线
%%%%%%%%%%%%%%%%%%%%%%%%%%%%%%%%%%%%%%%%%%%%%%%%%%%%%%%%%%%%%%%%%%%%%%%%%%%%%%%%%%%%%%%%%%%%%%%%%%%%%%%%%%%%%%%%%%%%%%%%%%%%%
                                &   1    &$1-2\sigma h+\sigma^2h^2$      &  $-2+2\sigma h$       &  1      &          &          \\
                                &   2    &$2-3\sigma h+\sigma^2h^2$    &  $-5+4\sigma h$    &$4-\sigma h$ & -1         &          \\
2 &   3    &$\frac{35}{12}-\frac{11}{3}\sigma h+\sigma^2h^2$   &  $-\frac{26}{3}+6\sigma h$       &$\frac{19}{2}-3\sigma h$
 & $-\frac{14}{3}+\frac{2}{3}\sigma h$  &    $\frac{11}{12} $   \\
%%%%%    %%%%%%%%
 &   4    &$\frac{15}{4}-\frac{25}{6}\sigma h+\sigma^2h^2$  & $-\frac{77}{6}+8\sigma h$
 & $\frac{107}{6}-6\sigma h$     &  $-13+\frac{8}{3}\sigma h$   & $\frac{61}{12}-\frac{1}{2}\sigma h$ &$-\frac{5}{6}$  \\\hline % 画底端的横线
    \end{tabular*}}\label{Table3.1}%\vspace{-15pt}
  \end{center}
\end{table}
\begin{table}[h]\fontsize{9.5pt}{12pt}\selectfont%生成浮动表格
  \begin{center}%\def\tabcolsep{28.5pt}%表格居中
  \caption{Coefficients $\overline{b}_p^{q,\sigma h}$  for approximations (\ref{3.3}) }\vspace{5pt}%标题，离表格一定的距离
 { \begin{tabular*}{\linewidth}{@{\extracolsep{\fill}}*{9}{c}}                                    \hline  %画顶端的横线
q &  $\nu$ & $x_N$             & $x_{N-1}$     & $x_{N-2}$  & $x_{N-3} $   &$x_{N-4} $  &  $x_{N-5} $  \\\hline
                                &   1    &$-1+\sigma h$      &  1        &        &          &          \\
                                &   2    &$-\frac{3}{2}+\sigma h$    &  2        & $-\frac{1}{2}$ &          &          \\
1 &   3    &$-\frac{11}{6}+\sigma h$   &  3        & $-\frac{3}{2}$ & $\frac{1}{3}$    &          \\
                                &   4    &$-\frac{25}{12}+\sigma h$  &  4        & -3     &  $\frac{4}{3}$   & $-\frac{1}{4}$  \\\hline %画底端的横线
%%%%%%%%%%%%%%%%%%%%%%%%%%%%%%%%%%%%%%%%%%%%%%%%%%%%%%%%%%%%%%%%%%%%%%%%%%%%%%%%%%%%%%%%%%%%%%%%%%%%%%%%%%%%%%%%%%%%%%%%%%%%%
                                &   1    &$1-2\sigma h+\sigma^2h^2$      &  $-2+2\sigma h$       &  1      &          &          \\
                                &   2    &$2-3\sigma h+\sigma^2h^2$    &  $-5+4\sigma h$    &$4-\sigma h$ & -1         &          \\
2 &   3    &$\frac{35}{12}-\frac{11}{3}\sigma h+\sigma^2h^2$   &  $-\frac{26}{3}+6\sigma h$       &$\frac{19}{2}-3\sigma h$
 & $-\frac{14}{3}+\frac{2}{3}\sigma h$  &    $\frac{11}{12} $   \\
%%%%%    %%%%%%%%
 &   4    &$\frac{15}{4}-\frac{25}{6}\sigma h+\sigma^2h^2$  & $-\frac{77}{6}+8\sigma h$
 & $\frac{107}{6}-6\sigma h$     &  $-13+\frac{8}{3}\sigma h$   & $\frac{61}{12}-\frac{1}{2}\sigma h$ &$-\frac{5}{6}$  \\\hline % 画底端的横线
    \end{tabular*}}\label{Table3.2}%\vspace{-15pt}
  \end{center}
\end{table}

For the Caputo fractional substantial derivative, we can rewrite it as
\begin{equation} \label{3.4}
\begin{split}
 {^s_c}{D}_t^\gamma G(x,\rho,t)
&={^s\!}D_t^\gamma [G(x,\rho,t)-e^{J\rho U(x) t}G(x,\rho,0)]\\
&={^s\!}D_t^\gamma \widetilde{G}(x,\rho,t)+ \sum_{q=1}^{m_1}  \frac{e^{J\rho U(x) t}t^{q-\gamma}}{\Gamma(q+1-\gamma)} \left({^s\!} D_t^{q}G(x,\rho,t)|_{t=0}\right),
\end{split}
\end{equation}
where $\widetilde{G}(x,\rho,t)=G(x,\rho,t)-e^{J\rho U(x) t}G(x,\rho,0)
  - \sum\limits_{q=1}^{m_1}\frac{t^{q} e^{J\rho U(x) t} }{\Gamma(q+1)} \left({^s\!} D_t^{q}G(x,\rho,t)|_{t=0}\right) $, $J=\sqrt{-1}$, and $m_1 \geq \nu-2$. Obviously, we have made sure that $\widetilde{G}(x,\rho,t)$ and its several derivatives w.r.t. $t$  at $t=0$ are zero; so the high order discretizations in time keep their high accuracy \cite{Chen:13}. This is the basic idea to obtain the high order approximations for the fractional derivative with nonzero initial conditions. The same idea will be used to treat the nonhomogeneous boundary conditions later.

It can be noted that two terms of the right hand side of (\ref{3.4}) automatically vanish when $m_1 \leq 0$.
And from (\ref{3.2}), we have
\begin{equation*}
\begin{split}
{^s\!} D_t^{q}G(x,\rho,t)|_{t=0}
&=\left(\frac{\partial}{\partial t}-J\rho  U(x)\right)^q G(x,\rho,t)|_{t=0} \\
& =\frac{1}{\tau^q}\!\sum_{p=0}^{q+\nu-1}b_p^{q,-J\rho U(x) \tau}G(x,\rho,t_p)+\mathcal{O}(\tau^\nu)\\
&=\frac{1}{\tau^q}\!\sum_{p=0}^{q+\nu-1}b_p^{q,-J\rho U_i \tau}G_{i,\rho}^p+\mathcal{O}(\tau^\nu).
\end{split}
\end{equation*}
Then
\begin{equation*}
\begin{split}
&{^s_c}D_t^\gamma G(x,\rho,t)|_{(x_i,t_n)}\\
&=\frac{1}{\tau^\gamma}\!\sum_{k=0}^{n}{d}_{i,k}^{\nu,\gamma}\! \left[G_{i,\rho}^{n-k}-e^{J\rho U_i t_{n-k}}G_{i,\rho}^{0}
  - \sum_{q=1}^{m_1}\frac{t_{n-k}^{q} e^{J\rho U_i t_{n-k}} }{\Gamma(q+1)}  \left(\frac{1}{\tau^q}\!\sum_{p=0}^{q+\nu-1}b_p^{q,-J\rho U_i \tau}G_{i,\rho}^p\right)  \right]\\
&\quad + \sum_{q=1}^{m_1}  \frac{e^{J\rho U_i t_n}t_n^{q-\gamma}}{\Gamma(q+1-\gamma)}  \left(\frac{1}{\tau^q}\!\sum_{p=0}^{q+\nu-1}b_p^{q,-J\rho U_i \tau}G_{i,\rho}^p\right)
+\mathcal{O}(\tau^\nu), ~~\nu=1,2,3,4.
\end{split}
\end{equation*}
Thus, the resulting discretization of (\ref{1.1}) can be rewritten as
\begin{equation}\label{3.5}
\begin{split}
& d_{i,0}^{\nu,\gamma}G_{i,\rho}^n-\kappa\sum_{j=0}^{M}\omega_{i,j}^{\alpha} G_{j,\rho}^n \\
&=\sum_{k=0}^{n-1}{d}_{i,k}^{\nu,\gamma}\! \left[e^{J\rho U_i t_{n-k}}G_{i,\rho}^{0}
 +  \sum_{q=1}^{m_1}\frac{t_{n-k}^{q} e^{J\rho U_i t_{n-k}} }{\Gamma(q+1)}  \left(\frac{1}{\tau^q}\!\sum_{p=0}^{q+\nu-1}b_p^{q,-J\rho U_i \tau}G_{i,\rho}^p\right)  \right]\\
&\quad -\sum_{k=1}^{n-1}d_{i,k}^{\nu,\gamma}G_{i,\rho}^{n-k}
  - \tau^\gamma  \sum_{q=1}^{m_1}  \frac{e^{J\rho U_i t_n}t_n^{q-\gamma}}{\Gamma(q+1-\gamma)}  \left(\frac{1}{\tau^q}\!\sum_{p=0}^{q+\nu-1}
  b_p^{q,-J\rho U_i \tau}G_{i,\rho}^p\right).
\end{split}
\end{equation}
There are two ways to solve the above equation: 1. use the idea of method of lines, but first we need to apply (\ref{2.31}) to obtain the starting values $G_{i,\rho}^p,\,p=1,\cdots,q+\nu-1$; 2. solve $G_{i,\rho}^p,\,p=1,\cdots,q+\nu-1$ in one times by inversing a big matrix.

%the values of $\widetilde{G}$ and its several derivatives are zero at the point $a$.
For the first term of  left Riemann-Liouville tempered fractional derivative (corresponding to (\ref{1.6})), it can be rewritten as
\begin{equation}\label{3.6}
 _{a}^1D_x^{\alpha,\lambda}G(x)={_{a}^1}D_x^{\alpha,\lambda} \widetilde{G}(x)
  + \sum_{q=0}^{m_2}  \frac{e^{-\lambda (x-a)}(x-a)^{q-\alpha}}{\Gamma(q+1-\alpha)} \left({_a^1}D_{x}^{q,\lambda}G(x)|_{x=a} \right) ,
\end{equation}
where $m_2\geq \nu-1$ and $\widetilde{G}(x)=G(x)
  - \sum\limits_{q=0}^{m_2}\frac{(x-a)^{q} e^{-\lambda (x-a)} }{\Gamma(q+1)} \left({_a^1}D_{x}^{q,\lambda}G(x)|_{x=a} \right)$. It is clear that $\widetilde{G}(x)$ and its several derivatives are equal to zero at the left boundaries. So its high order discretizations can keep their high accuracy \cite{Chen:1313,Chen:0013}. In fact, the various high order discretizations, say, being derived from the so-called WSLD operators \cite{Chen:1313,Chen:0013}, can be uniformly written as
\begin{equation}\label{3.7}
 \begin{split}
_{a}^1D_{x}^{\alpha,\lambda} \widetilde{G}(x_i)
&=\frac{1}{h^{\alpha}}\sum_{j=0}^{i+1}l_j^\alpha \widetilde{G}(x_{i-j+1})+\mathcal{O}(h^\nu).
 \end{split}
\end{equation}
From (\ref{3.2}), (\ref{3.6}),  and (\ref{3.7}), we obtain
\begin{equation}\label{3.8}
 \begin{split}
_{a}^1D_{x}^{\alpha,\lambda}G(x_i)&=  \frac{1}{h^{\alpha}}\sum_{j=0}^{i+1}l_j^{\alpha}\widetilde{G}(x_{i-j+1})
\\
& \quad + \sum_{q=0}^{m_2}  \frac{e^{-\lambda (x_i-a)}(x_i-a)^{q-\alpha}}{\Gamma(q+1-\alpha)}  \left( \frac{1}{h^q}\sum_{p=0}^{q+\nu-1}b_p^{q,\lambda h}G(x_p)\right)+\mathcal{O}(h^\nu)\\
&=\frac{1}{h^{\alpha}}\sum_{j=0}^{i+1}{_L}\widetilde{l}_j^\alpha G(x_{i-j+1})+\mathcal{O}(h^\nu),~~i=1,2,\ldots M-1,
 \end{split}
\end{equation}
where $\widetilde{G}(x_{i-j+1}) = G(x_{i-j+1}) - \sum\limits_{q=0}^{m_2}\frac{(x_{i-j+1}-a)^{q} e^{-\lambda (x_{i-j+1}-a)} }{\Gamma(q+1)} \left( \frac{1}{h^q}\sum\limits_{p=0}^{q+\nu-1}b_p^{q,\lambda h}G(x_p)\right)$.

 When $\nu=2$, $m_2=1$, from (\ref{3.8}) we have
\begin{equation}\label{3.9}
\begin{split}
{_L}\widetilde{l}_{i+1}^\alpha
&=l_{i+1}^\alpha  -\sum\limits_{j=0}^{i+1}l_j^\alpha \left\{\frac{e^{-\lambda (x_{i-j+1}-a)}}{h}
\left[h+\left(\lambda h-\frac{3}{2}\right) (x_{i-j+1}-a)\right]  \right\}  \\
&\quad+h^{\alpha-1} \frac{e^{-\lambda (x_i-a)} (x_i-a)^{-\alpha}}{\Gamma(2-\alpha)} \left[(1-\alpha)h+\left(\lambda h-\frac{3}{2}\right) (x_i-a)\right] ;  \\
{_L}\widetilde{l}_{i}^\alpha
&=l_{i}^\alpha  -\sum\limits_{j=0}^{i+1}l_j^\alpha \left(\frac{e^{-\lambda (x_{i-j+1}-a)}}{h}
2 (x_{i-j+1}-a)  \right) +h^{\alpha-1} \frac{e^{-\lambda (x_i-a)} (x_i-a)^{-\alpha}}{\Gamma(2-\alpha)} 2 (x_i-a) ;  \\
{_L}\widetilde{l}_{i-1}^\alpha
&=l_{i-1}^\alpha  -\sum\limits_{j=0}^{i+1}l_j^\alpha \left[\frac{e^{-\lambda (x_{i-j+1}-a)}}{h}
\left(-\frac{1}{2} (x_{i-j+1}-a) \right)  \right] \\
&\quad +h^{\alpha-1} \frac{e^{-\lambda (x_i-a)} (x_i-a)^{-\alpha}}{\Gamma(2-\alpha)} \left(-\frac{1}{2} (x_i-a) \right) ;   \\
{_L}\widetilde{l}_{j}^\alpha
&=l_{j}^\alpha, ~~j \neq  i-1, i, i+1.
\end{split}
\end{equation}

%%%%%%%%%%%%%%%%%%%%
For the first term of  right Riemann-Liouville tempered fractional derivative (corresponding to (\ref{1.6})), it can be written as
\begin{equation}\label{3.10}
\begin{split}
 _{x}^1D_b^{\alpha,\lambda}G(x)
 ={_{x}^1}D_b^{\alpha,\lambda}\bar{G}(x)+ \sum_{q=0}^{m_2}  \frac{e^{-\lambda (b-x)}(b-x)^{q-\alpha}}{\Gamma(q+1-\alpha)} \left({_x^1}D_{b}^{q,\lambda}G(x)|_{x=b} \right) ,
\end{split}
\end{equation}
where $m_2\geq \nu-1$ and $\bar{G}(x)=G(x)
  - \sum\limits_{q=0}^{m_2}\frac{(b-x)^{q} e^{-\lambda (b-x)} }{\Gamma(q+1)}  \left({_x^1}D_{b}^{q,\lambda}G(x)|_{x=b} \right)$. It is obvious that $\bar{G}(x)$ and its several derivatives are equal to zero at the right boundaries. So its high order discretizations can keep their high accuracy \cite{Chen:1313,Chen:0013}. The general discretizations of $\bar{G}(x)$ can be written as
\begin{equation}\label{3.11}
 \begin{split}
_{x}^1D_{b}^{\alpha,\lambda} \bar{G}(x_i)
&=\frac{1}{h^{\alpha}}\sum_{j=0}^{M-i+1}l_j^\alpha \bar{G}(x_{i+j-1})+\mathcal{O}(h^\nu).
 \end{split}
\end{equation}
From (\ref{3.3}), (\ref{3.10}), and (\ref{3.11}), there exists
\begin{equation}\label{3.12}
 \begin{split}
&_{x}^1D_{b}^{\alpha,\lambda}G(x_i)\\
&=  \frac{1}{h^{\alpha}}\!\!\sum_{j=0}^{M-i+1}\!\!l_j^{\alpha}
\Bigg[  G(x_{i+j-1})  \\
&\quad - \sum_{q=0}^{m_2}\frac{(b-x_{i-j+1})^{q} e^{-\lambda (b-x_{i-j+1})} }{\Gamma(q+1)}
\left(\frac{1}{h^q}\sum_{p=x_N-q-\nu+1}^{x_N}\overline{b}_p^{q,\lambda h}G(x_p)\right)    \Bigg]\\
& \quad+ \sum_{q=0}^{m_2}  \frac{e^{-\lambda (b-x_i)}(b-x_i)^{q-\alpha}}{\Gamma(q+1-\alpha)}
\left(\frac{1}{h^q}\sum_{p=x_N-q-\nu+1}^{x_N}\overline{b}_p^{q,\lambda h}G(x_p)\right)  +\mathcal{O}(h^\nu)\\
&=\frac{1}{h^{\alpha}}\sum_{j=0}^{M-i+1}{_R}\widetilde{l}_j^\alpha G(x_{i+j-1})+\mathcal{O}(h^\nu),~~i=1,2,\ldots M-1.
 \end{split}
\end{equation}
Hence, for $\nu=2$, $m_2=1$, from (\ref{3.12}) we have
\begin{equation}\label{3.13}
\begin{split}
{_R}\widetilde{l}_{M-i+1}^\alpha
&=l_{M-i+1}^\alpha  -\sum\limits_{j=0}^{M-i+1}l_j^\alpha \left\{\frac{e^{-\lambda (b-x_{i+j-1})}}{h}
\left[h+\left(\lambda h-\frac{3}{2}\right) (b-x_{i+j-1})\right]  \right\} \\
&\quad+h^{\alpha-1} \frac{e^{-\lambda (b-x_i)} (b-x_i)^{-\alpha}}{\Gamma(2-\alpha)} \left[(1-\alpha)h+\left(\lambda h-\frac{3}{2}\right) (b-x_i)\right] ;  \\
{_R}\widetilde{l}_{M-i}^\alpha
&=l_{M-i}^\alpha  -\sum\limits_{j=0}^{M-i+1}l_j^\alpha \left(\frac{e^{-\lambda (b-x_{i+j-1})}}{h}2 (b-x_{i+j-1})  \right)\\
&\quad +h^{\alpha-1} \frac{e^{-\lambda (b-x_i)} (b-x_i)^{-\alpha}}{\Gamma(2-\alpha)} 2 (b-x_i) ;  \\
{_R}\widetilde{l}_{M-i-1}^\alpha
&=l_{M-i-1}^\alpha  -\sum\limits_{j=0}^{M-i+1}l_j^\alpha \left[\frac{e^{-\lambda (b-x_{i+j-1})}}{h}
\left(-\frac{1}{2} (b-x_{i+j-1}) \right)  \right] \\
&\quad +h^{\alpha-1} \frac{e^{-\lambda (b-x_i)} (b-x_i)^{-\alpha}}{\Gamma(2-\alpha)} \left(-\frac{1}{2} (b-x_i) \right) ;   \\
{_R}\widetilde{l}_{j}^\alpha
&=l_{j}^\alpha, ~~j \neq  M-i-1, M-i, M-i+1.
\end{split}
\end{equation}
Therefore, we obtain the general high order scheme of (\ref{1.1}) with nonhomogeneous boundary and/or initial conditions:
\begin{equation}\label{3.14}
\begin{split}
& d_{i,0}^{\nu,\gamma}G_{i,\rho}^n-\kappa\sum_{j=0}^{M}l_{i,j}^{\alpha} G_{j,\rho}^n \\
&=\sum_{k=0}^{n-1}{d}_{i,k}^{\nu,\gamma}\! \left[e^{J\rho U_i t_{n-k}}G_{i,\rho}^{0}
 +  \sum_{q=1}^{m_1}\frac{t_{n-k}^{q} e^{J\rho U_i t_{n-k}} }{\Gamma(q+1)} \left(\frac{1}{\tau^q}\!\sum_{p=0}^{q+\nu-1}b_p^{q,-J\rho U_i \tau}G_{i,\rho}^p\right)   \right]\\
&\quad -\sum_{k=1}^{n-1}d_{i,k}^{\nu,\gamma}G_{i,\rho}^{n-k}
 - \tau^\gamma  \sum_{q=1}^{m_1}  \frac{e^{J\rho U_i t_n}t_n^{q-\gamma}}{\Gamma(q+1-\gamma)} \left(\frac{1}{\tau^q}\!\sum_{p=0}^{q+\nu-1}b_p^{q,-J\rho U_i \tau}G_{i,\rho}^p\right).
\end{split}
\end{equation}
In the numerical computations of next subsection, we take
\begin{equation*}
\begin{split}
&l_{0}^\alpha=r_1g_{0}^\alpha; ~~~~
l_{1}^\alpha=r_1g_{1}^\alpha+r_2g_{0}^\alpha;~~~~l_{j}^\alpha=r_1g_{j}^\alpha+r_2g_{j-1}^\alpha+r_3g_{j-2}^\alpha, ~~\forall j\geq 2;
\end{split}
\end{equation*}
and from  (\ref{3.9}) and (\ref{3.13}), there exists
\begin{equation*}
l_{i,j}^{\alpha}=\left\{ \begin{array}
 {l@{\quad } l}
  e^{-(i-j)\lambda h}{_L}\widetilde{l}_{i-j+1}^{\alpha},&j < i-1,\\
 e^{\lambda h} {_R}\widetilde{l}_{0}^{\alpha}+e^{-\lambda h}{_L}\widetilde{l}_{2}^{\alpha} ,&j=i-1,\\
{_L}\widetilde{l}_{1}^{\alpha}+ {_R}\widetilde{l}_{1}^{\alpha}-2\left(r_1e^{\lambda h}+r_2+r_3e^{-\lambda h}\right)\left(1-e^{-\lambda h}\right)^{\alpha},&j=i,\\
 e^{\lambda h}{_L}\widetilde{l}_{0}^{\alpha}+e^{-\lambda h}{_R}\widetilde{l}_{2}^{\alpha} ,&j=i+1,\\
 e^{-(j-i)\lambda h}{_R}\widetilde{l}_{j-i+1}^{\alpha} ,&j>i+1\\
 \end{array}
 \right.
\end{equation*}
with $i=1,\ldots,M-1$. The numerical results will show that (\ref{3.14}) truly has high convergence orders for (\ref{1.1}) with {\em nonhomogeneous} boundary and/or initial conditions.
\begin{remark}
For $\lambda=0$, the scheme (\ref{3.14}) reduces to the high order scheme for the backward fractional Feynman-Kac equation with L\'{e}vy flight \cite{Carmi:10} and nonhomogeneous boundary and/or initial conditions. For $\lambda=0$ and $\rho=0$,  the scheme (\ref{3.14}) becomes the high order scheme for the space-time Caputo-Riesz fractional diffusion equation \cite{Chen:12} with nonhomogeneous boundary and/or initial conditions.
\end{remark}
\begin{remark}
The scheme (\ref{3.14}) makes a breakthrough in the requirement of using the nodes within the bounded interval (Remark 3 of \cite{Tian:12}), i.e., if the nodes are beyond the bounded interval, just simply let the values of the function at nodes be zero; the high convergence orders still remain.
\end{remark}
\begin{remark}
Using the iterative methods to solve the algebraic equations corresponding to the scheme (\ref{3.14}), the cost of matrix-vector multiplication still keeps as $\mathcal{O} (M \log M)$ since the Toeplitz structure of the matrix still remains except several fixed boundary columns.
\end{remark}

%%%%%%%%%%%%%%%%%%%%%%%%%%  Example 5.3  %%%%%%%%%%%%%%%%%%%%%
\subsection{High numerical convergence orders and physical simulations}
Without loss of generality, we add a force term $f(x,\rho,t)$ on the right side of (\ref{1.1}).
The following numerical results show that the scheme (\ref{3.14}) has high convergence orders for (\ref{1.1}) with nonhomogeneous boundary and/or initial conditions.

\begin{example}\label{example2}\end{example}
Consider  (\ref{1.1}) on a finite domain  $x \in (0,1)$ and $0<t \leq 1/2$  with the coefficient  $K=1$ and $U(x)=x$, $\rho=1$.
%\begin{equation*}
%\begin{split}
%f(x,\rho,t)=&e^{J\rho x t}\left[ \frac{\Gamma(4+\gamma)}{\Gamma(4)}t^3  +\frac{\Gamma(3)}{\Gamma(3-\gamma)}t^{2-\gamma}  +\frac{\Gamma(2)}{\Gamma(2-\gamma)}t^{1-\gamma}  \right]\sin(x^2)\sin((1-x)^2)  \\
%       &  +\frac{1}{2\cos(\alpha \pi/2)} (t^{3+\gamma}+t^2+t+1)e^{-\lambda x}{ _{a}}D_x^{\alpha}[e^{(\lambda+J\rho t) x}\sin(x^2)\sin((1-x)^2)]\\
%       &  +\frac{1}{2\cos(\alpha \pi/2)} (t^{3+\gamma}+t^2+t+1)e^{\lambda x}{ _{x}}D_b^{\alpha}[e^{(-\lambda+J\rho t) x}\sin(x^2)\sin((1-x)^2)],
%\end{split}
%\end{equation*}
%where $J=\sqrt{-1}$.
Let the nonhomogeneous initial condition $G(x,\rho,0)=\sin(x^2)\sin((1-x)^2) $ and the boundary
conditions $G(0,\rho,t)=G(1,\rho,t)=0$. Taking the exact solution as $$G(x,\rho,t)=e^{J\rho xt}(t^{3+\gamma}+t^3+t^2+t+1)\sin(x^2)\sin((1-x)^2), ~~J=\sqrt{-1}, $$
it is easy to analytically get the forcing function $f(x,\rho,t)$.

\begin{table}[h]\fontsize{9.5pt}{12pt}\selectfont%生成浮动表格
  \begin{center}%\def\tabcolsep{28.5pt}%表格居中
  \caption{The maximum errors and convergence orders for (\ref{3.5}) with $\nu=4$ at $t=1/2$,  and
$\lambda=0.2$,  $U(x)=x$, $\rho=1$,   $r_3= 0$,  $m_1=2$, $h=\tau^2.$ }\vspace{5pt}%标题，离表格一定的距离
 { \begin{tabular*}{\linewidth}{@{\extracolsep{\fill}}*{9}{c}}                                    \hline  %画顶端的横线
$\tau$ &  $\alpha=1.1,\gamma=0.9$  & Rate & $\alpha=1.5,\gamma=0.5$     & Rate  & $\alpha=1.9,\gamma=0.1 $&   Rate      \\\hline
    $1/10$   &    3.2798e-005  &           & 2.8455e-005     &       &  3.9471e-005   &          \\\hline %画底端的横线
    $1/20$   &   2.3499e-006 & 3.80   &  2.0029e-006   & 3.83       &  2.3969e-006   &  4.04    \\\hline %画底端的横线
    $1/40$   &   1.6838e-007  &3.80   &  1.4330e-007    & 3.81       &   1.4513e-007   &  4.05 \\\hline %画底端的横线
    $1/80$   &   1.1480e-008  &  3.87   &  9.8246e-009   & 3.87       &  8.7080e-009  & 4.06 \\\hline %画底端的横线
    \end{tabular*}}\label{table:3}%\vspace{-15pt}
  \end{center}
\end{table}

Table \ref{table:3}  shows that the scheme (\ref{3.5}) used to solve Example \ref{example2} (with nonhomogeneous initial condition) preserves the desired convergence order $\mathcal{O} (\tau^\nu+h^2)$, $\nu=4$.

\begin{example}\label{example3}\end{example}
Consider (\ref{1.1}) on a finite domain  $x \in (0,1)$, $0<t \leq 1/2$  with the coefficient $K=1$, $U(x)=x$, and $\rho=1$.
Let the nonhomogeneous initial condition $G(x,\rho,0)=x^4-2x^3-x^2+2x+1 $ and the nonhomogeneous boundary
conditions $G(0,\rho,t)=t^{3+\gamma}+t^3+t^2+t+1$ and $G(1,\rho,t)=e^{J\rho t}(t^{3+\gamma}+t^3+t^2+t+1)$. Taking the exact solution as
$$G(x,\rho,t)=e^{J\rho xt}(t^{3+\gamma}+t^3+t^2+t+1)(x^4-2x^3-x^2+2x+1), ~~J=\sqrt{-1}, $$
we get the forcing function
\begin{equation*}
\begin{split}
&f(x,\rho,t)\\
&=e^{J\rho x t}\Big[\frac{\Gamma(4+\gamma)}{\Gamma(4)}t^3 +\frac{\Gamma(4)}{\Gamma(4-\gamma)}t^{3-\gamma}
+\frac{\Gamma(3)}{\Gamma(3-\gamma)}t^{2-\gamma} \\
&~~+\frac{\Gamma(2)}{\Gamma(2-\gamma)}t^{1-\gamma}
+\frac{\Gamma(1)}{\Gamma(1-\gamma)}t^{-\gamma}   \Big](x^4-2x^3-x^2+2x+1)\\
       & ~~ -\frac{\lambda^\alpha}{\cos(\alpha \pi/2)}e^{J\rho xt}(t^{3+\gamma}+t^3+t^2+t+1)(x^4-2x^3-x^2+2x+1)\\
       &~~  +\frac{(t^{3+\gamma}+t^3+t^2+t+1)}{2\cos(\alpha \pi/2)} e^{-\lambda x}{ _{a}}D_x^{\alpha}[e^{(\lambda+J\rho t) x}(x^4-2x^3-x^2+2x+1)]\\
       &~~  +\frac{(t^{3+\gamma}+t^3+t^2+t+1)}{2\cos(\alpha \pi/2)} e^{\lambda x} \\
       &~~~~\times{ _{x}}D_b^{\alpha}[e^{(-\lambda+J\rho t) x}
       ((1-x)^4-2(1-x)^3-(1-x)^2+2(1-x)+1)]
\end{split}
\end{equation*}
To calculate the last two terms of the above equation,
the following fractional formulas \cite{Miller:93} are used:
%\begin{equation*}
%\begin{split}
%&_{0}D_x^{\alpha}[x^m ]=\frac{\Gamma(m+1)}{\Gamma(m+1-\alpha)}x^{m-\alpha},~~~~_{x}D_b^{\alpha}[(b-x)^m ]=\frac{\Gamma(m+1)}{\Gamma(m+1-\alpha)}(b-x)^{m-\alpha};
% \end{split}
%\end{equation*}
%and
\begin{equation*}
\begin{split}
_{a}D_x^{\alpha}[(x-a)^m e^{\lambda (x-a)}]&=\sum_{n=0}^\infty \frac{\lambda^n}{n!}\frac{\Gamma(n+m+1)}{\Gamma(n+m+1-\alpha)}(x-a)^{n+m-\alpha}
\\
&\simeq \sum_{n=0}^{50} b_{m,n}(x-a)^{n+m-\alpha},
 \end{split}
\end{equation*}
where $b_{m,0}=\frac{\Gamma(m+1)}{\Gamma(m+1-\alpha)}$, $b_{m,k}=\frac{\lambda(m+n)}{n(m+n-\alpha)}b_{m,k-1}, ~k \geq 1;$ and
\begin{equation*}
\begin{split}
_{x}D_b^{\alpha}[(b-x)^m e^{\lambda x}]
&=e^{\lambda b}\sum_{n=0}^\infty (-1)^n\frac{\lambda^n}{n!}\frac{\Gamma(n+m+1)}{\Gamma(n+m+1-\alpha)}(b-x)^{n+m-\alpha}\\
&\simeq (-1)^n e^{\lambda b}\sum_{n=0}^{50} b_{m,n}x^{n+m-\alpha}.
 \end{split}
\end{equation*}
\begin{table}[h]\fontsize{9.5pt}{12pt}\selectfont%生成浮动表格
  \begin{center}%\def\tabcolsep{28.5pt}%表格居中
  \caption{The maximum errors and convergence orders for (\ref{3.14}) with $\nu=2$ at $t=1/2$,  and
  $U(x)=x$, $\rho=1$,   $r_3= 0$, $m_1=2$, $h=\tau.$ }\vspace{5pt}% 标题，离表格一定的距离
 { \begin{tabular*}{\linewidth}{@{\extracolsep{\fill}}*{9}{c}}                                    \hline  %画顶端的横线
       &  $\lambda=0.2$  &    & $\lambda=1 $&  & $\lambda=5 $&       \\\hline
      $\tau$ &  $\alpha=1.3,\gamma=0.8$  & Rate & $\alpha=1.5,\gamma=0.5$  & Rate   & $\alpha=1.9,\gamma=0.2 $&   Rate      \\\hline
    $1/20$   &    3.8162e-003            &      &    3.2514e-003                      &        &  4.1398e-002   &          \\\hline %画底端的横线
    $1/40$   &    1.0620e-003            &1.85  &    8.4401e-004                      &  1.95      &  1.0920e-002   &  1.92    \\\hline %画底端的横线
    $1/80$   &   2.8439e-004             &1.90  &    2.1568e-004                      &   1.97     &   2.8053e-003   &  1.96 \\\hline %画底端的横线
    $1/160$  &   7.6409e-005             & 1.90 &    5.4627e-005                      &   1.98     & 7.1366e-004 & 1.97  \\\hline %画底端的横线
    \end{tabular*}}\label{table:4}%\vspace{-15pt}
  \end{center}
\end{table}

Table \ref{table:4} shows that the scheme (\ref{3.14}) preserves the desired convergence order $\mathcal{O} (\tau^2+h^2)$ for Example \ref{example3} with nonhomogeneous boundary and initial conditions.

%%%%%%%%%%%%%%%%%%%%%%%%%%%%%%%%%%%%%%%%%%%%%
%%%%%%%%%%%%%%%%%%%%%%%%%%%%%%%%%%%%%%%%%%%%%%%

{\bf{Simulations with Dirac delta function as initial condition}}
Let the joint PDF $G(x,A,t)$ be the inverse Fourier transform $\rho\rightarrow A$ of $G(x,\rho,t)$ \cite{Carmi:10}.

Simulate (\ref{3.14}) on a finite domain  $0< x < 1 $,  $0<t \leq 1$,  with the coefficient   $K=1$,  $\nu=2$, $\lambda=0.1$,  $r_3=0$, $m_1=0$, $\tau=h=1/100$, and $\alpha=1.5$, $\gamma=0.5$,
\begin{equation*}
U(x)=\left\{ \begin{array}
{l@{\quad} l}1,~~~~x \in (0.25,0.75),\\
  0,~~~~{\rm otherwise}.
\end{array}
\right.
\end{equation*}

Let the initial condition $G(x,\rho,0)=\delta_a(x-0.5)$ (Dirac delta function) and the boundary conditions $G(0,\rho,t)=G(1,\rho,t)=0$,
where $\rho=\{k\}_{k=-39}^{40}$. The Dirac delta function is approximated by the sequence of Gaussian functions
\begin{equation*}
  \delta_a(x)=\frac{1}{2\sqrt{a\pi}}e^{-\frac{x^2}{4a}}~~~~{\rm as}~~~~a\rightarrow 0.
\end{equation*}
Here we take $a=0.001$ as the approximation in numerical computations.

The corresponding  procedure of generating Figures \ref{FIG.1}-\ref{FIG.4} is executed as follows:
\begin{description}
\item[(1)] Using the scheme (\ref{3.14}) and the above conditions, we obtain $G(x,\rho,t)$.
\item[(2)] From the {\bf{Inverse discrete Fourier transform (IDFT) method}} \cite{{Smith:99}}, we get $G(x,A,t)$.
\end{description}

%%%%%%%% Figure  %%%%%%%%%%%%%%%%%%%%%%%%%%%%%%%%%%%%%%%%%%   Figure      %%%%%%%%%%%%%%%%%%%%%%%%%   Figure      %%%%%%%%%%%%%%%%%%%%%%%%%%%%%%%%%
%%%%%%%%%%%%%%%%%%%%%%%%%%%%%%%%%%%%%%%%%%%%%%%%%%%
\begin{figure}[t]
%%%%%%%%%%
    \begin{minipage}[t]{0.50\linewidth}
    \includegraphics[scale=0.45]{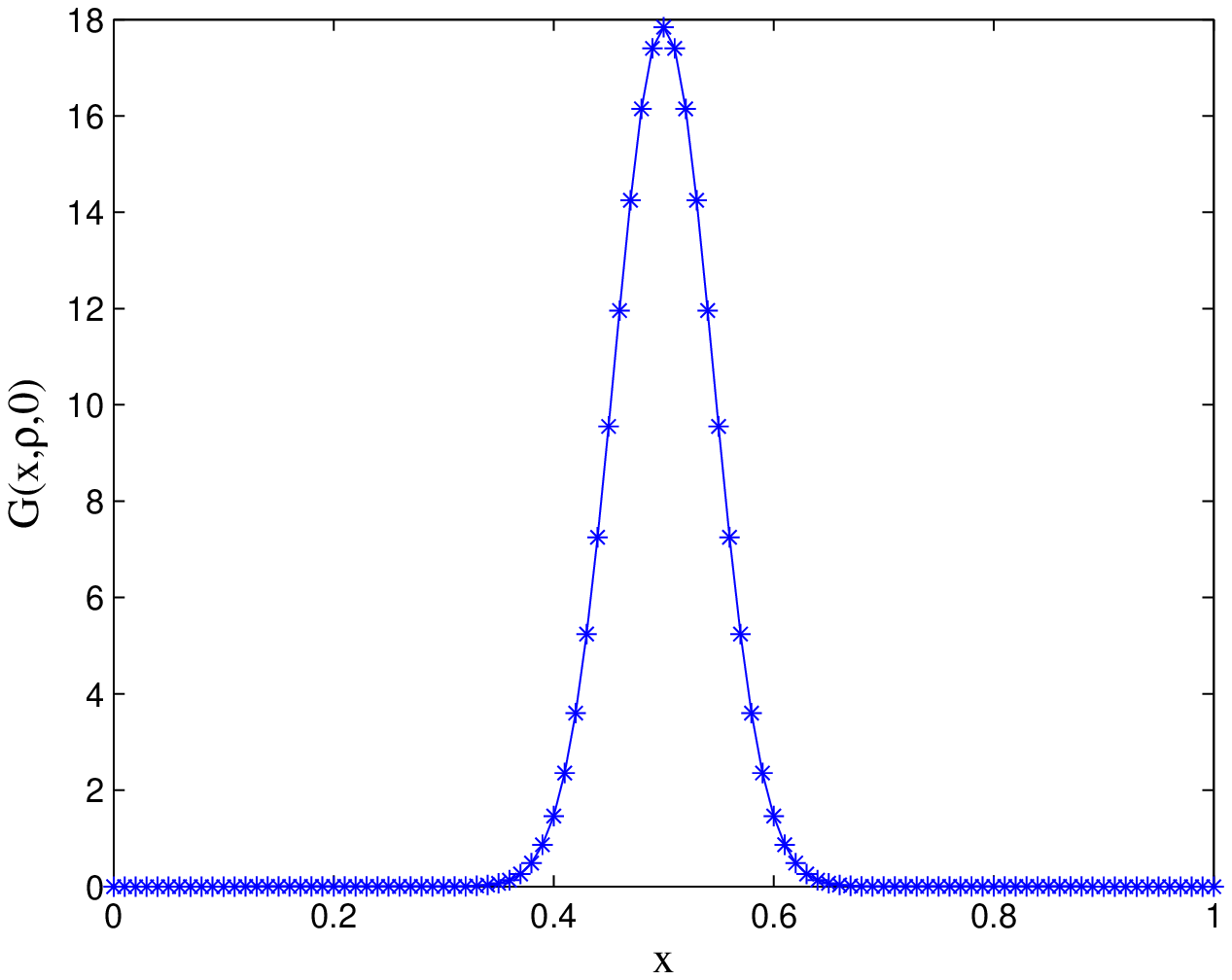}
    \caption{Initial value.}  \label{FIG.1}
    \end{minipage}
%%%%%%%%%%
    \begin{minipage}[t]{0.50\linewidth}
    \includegraphics[scale=0.45]{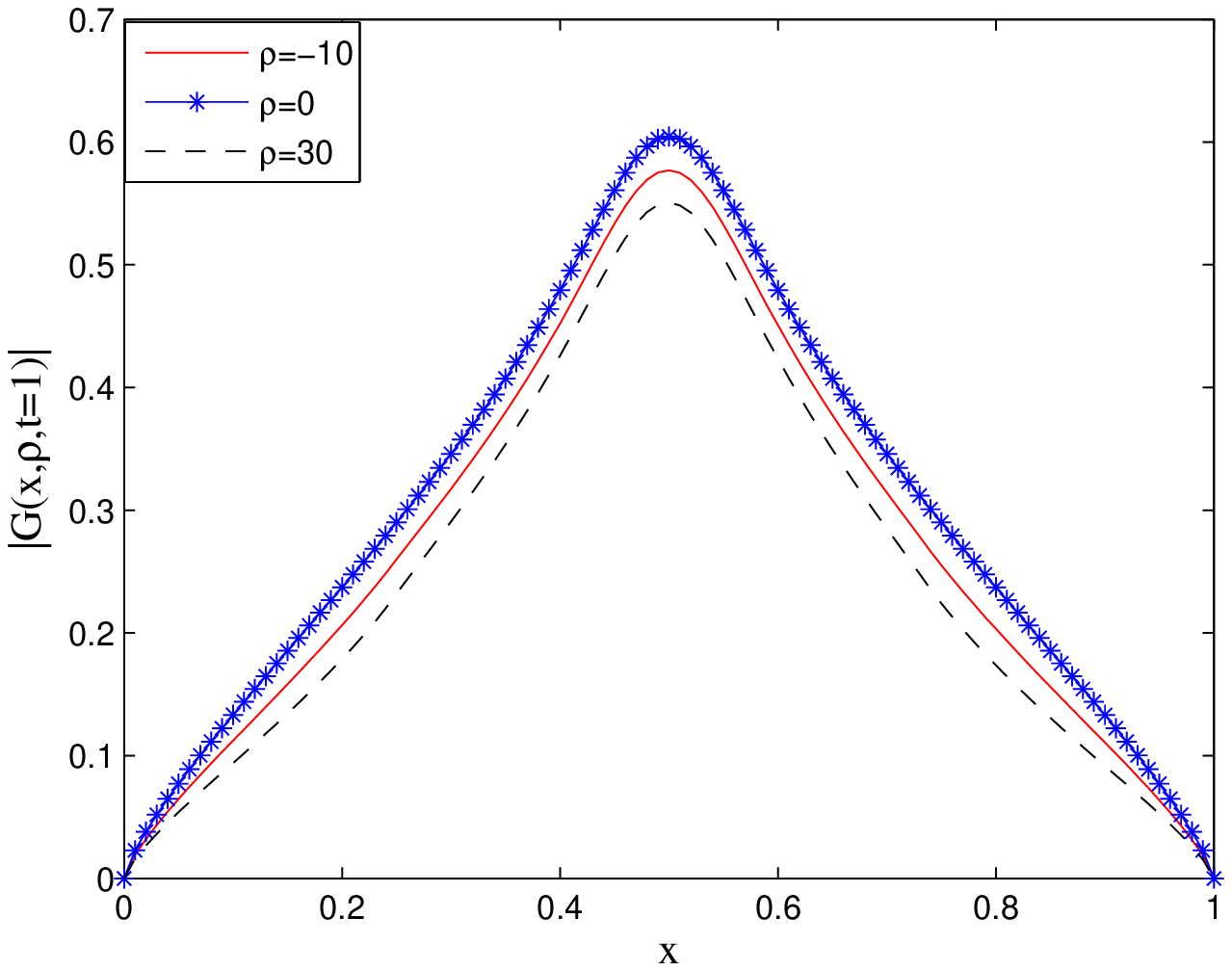}
    \caption{Amplitude $G(x,\rho,t=1)$.}  \label{FIG.2}
    \end{minipage}
%%%%%%%%%%
    \begin{minipage}[t]{0.50\linewidth}
    \includegraphics[scale=0.45]{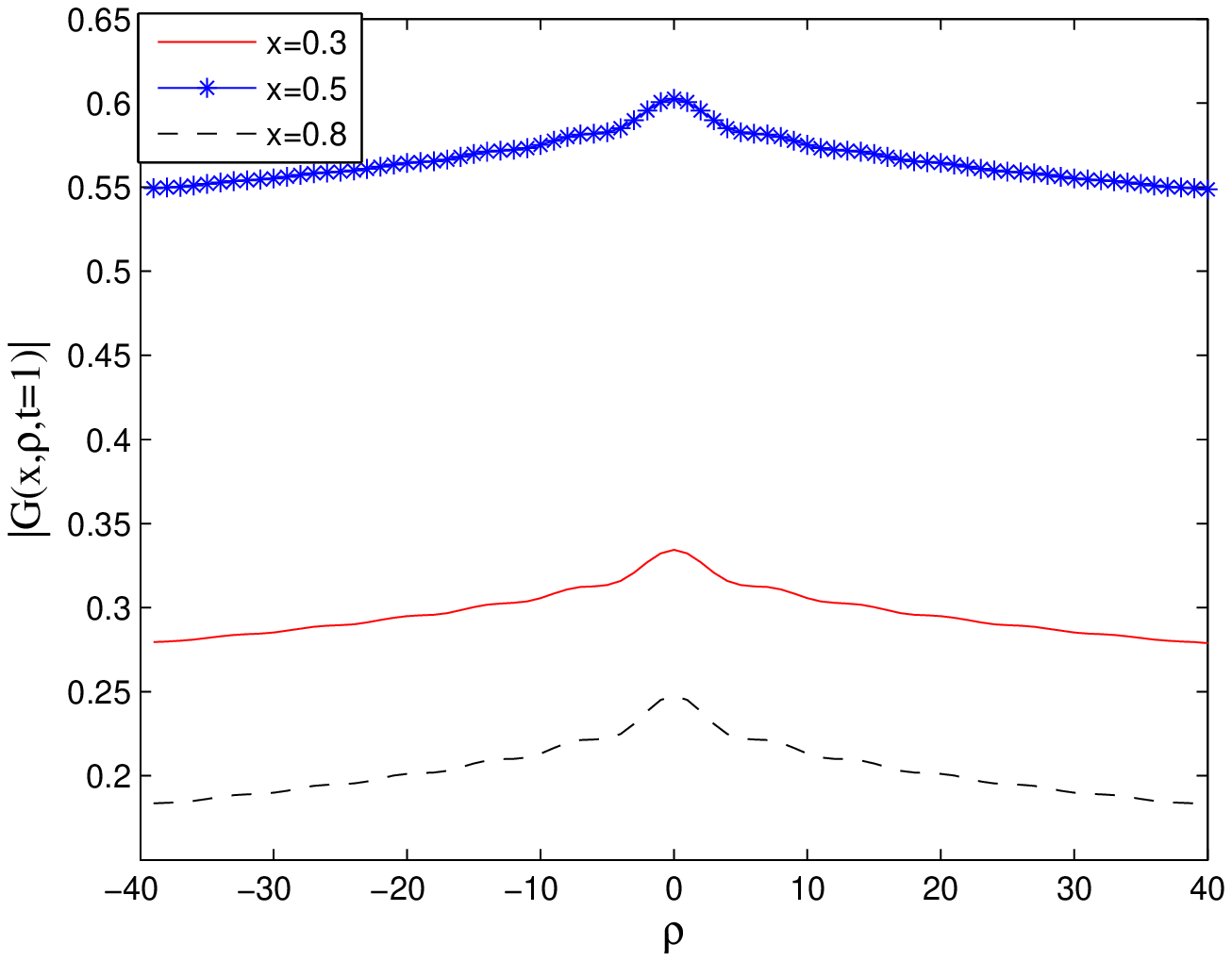}
    \caption{Amplitude $G(x,\rho,t=1)$.}  \label{FIG.3}
    \end{minipage}
%%%%%%%%%%
    \begin{minipage}[t]{0.50\linewidth}
    \includegraphics[scale=0.45]{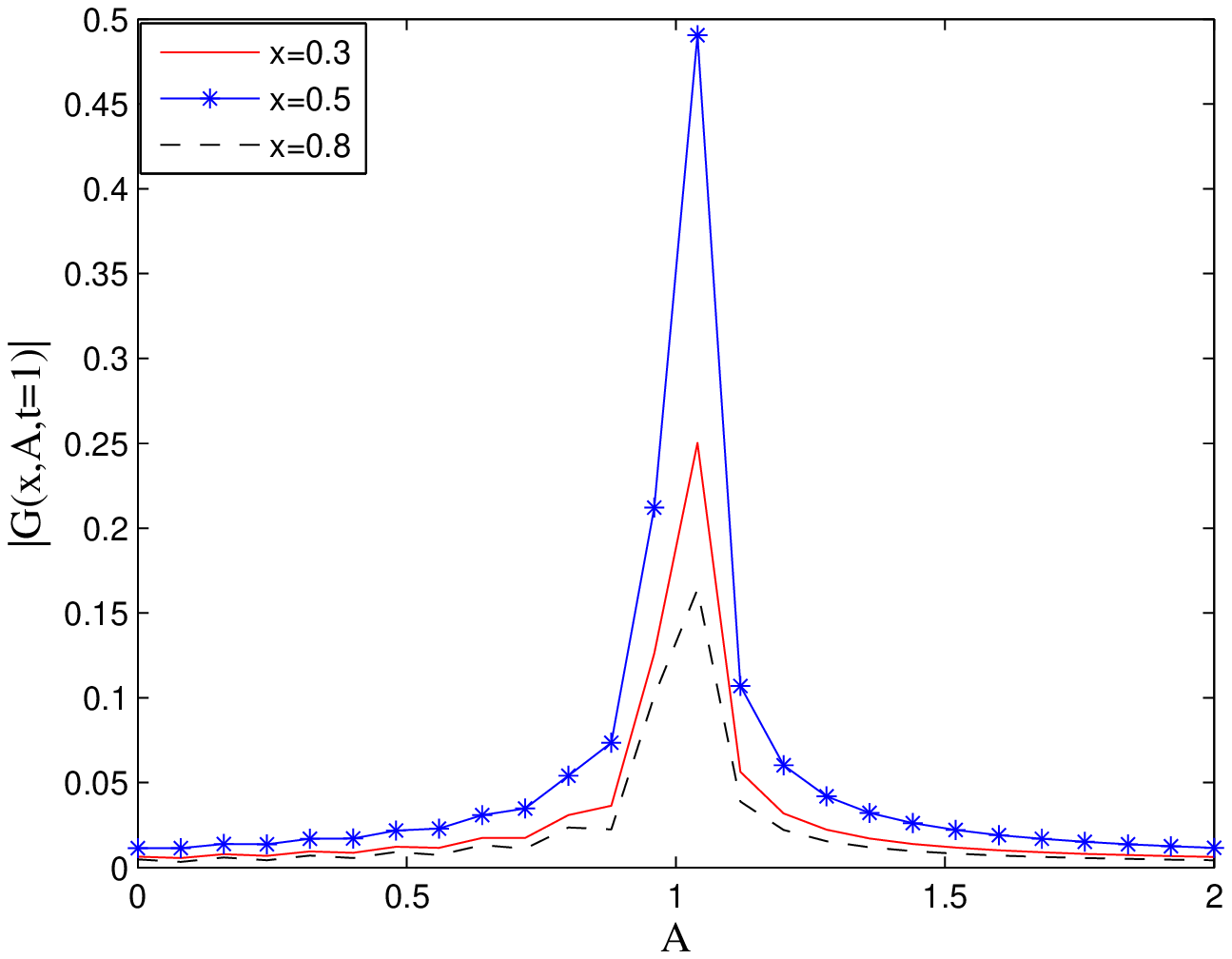}
    \caption{IDFT of $G(x,\rho,t=1)$.}  \label{FIG.4}
    \end{minipage}

\end{figure}

%%%%%%%%%% Figure  %%%%%%%%%%%%%%%%%%%%%%%%%%%%%%%%%%%%%%%%%%   Figure      %%%%%%%%%%%%%%%%%%%%%%%%%   Figure      %%%%%%%%%%%%%%%%%%%%%%%%%%%%%%%%%
%%%%%%%% Figure  %%%%%%%%%%%%%%%%%%%%%%%%%%%%%%%%%%%%%%%%%%   Figure      %%%%%%%%%%%%%%%%%%%%%%%%%   Figure      %%%%%%%%%%%%%%%%%%%%%%%%%%%%%%%%%
Figures \ref{FIG.1}-\ref{FIG.4} show the simulation results with Dirac delta function as initial condition, and we find the joint probability density function $G(x,A,t)$.

\section{Conclusion}
In Fourier space, fractional substantial diffusion equation with truncated L\'{e}vy flights describes the distribution of the functionals of the paths generated by the particles with the power law waiting time distribution $t^{-(1+\gamma)}$ with $0<\gamma<1$ and the jump length distribution $e^{-\lambda |x|}|x|^{-1-\alpha}$ with $\lambda>0$. For the equation, this paper first provides its numerical scheme with first order accuracy in time and second order in space and the detailed numerical stability and convergence analysis are performed in complex space;  then we further propose the numerical schemes with high accuracy in both time and space for the equation with nonhomogeneous boundary and/or initial conditions. The high order schemes for the fractional problems with nonhomogeneous boundary and/or initial conditions make a breakthrough in two aspects:  1. breaking the requirements that the analytical solution and even its several derivatives must be zero at the boundaries and/or initial time for keeping high accuracy; 2. breaking the limitation of using the nodes within the bounded interval (just simply putting the values of the function at nodes beyond the interval be zero). The extensive numerical experiments are made by multigrid methods, which numerically verify the high convergence orders and simulate the physical system by numerically making inverse Fourier transform to the solutions of the equation with different $\rho$. Taking the parameters $\rho=0$ and $\lambda=0$, all the schemes discussed in this paper become to the schemes for the space-time Caputo-Riesz fractional diffusion equation with homogeneous/nonhomogeneous boundary and homogeneous/nonhomogeneous initial conditions, which are still effective and keep the high accuracy.

\section*{Acknowledgments}The authors thank Yantao Wang for the discussions.

\section*{Appendixes}
We provide the pseudo codes used in this paper.
%\label{sec:V-cycle}
%For a general linear system
%\begin{equation*}
%A_hu_h=b_h,
%\end{equation*}
%the following V-cycle MGM  can be used to solve.
\begin{algorithm}
\caption{MGM ~~~~$u_h$=V-cycle$(A_h,u_0,b_h)$ }
\label{V-cycle}
\begin{algorithmic}
\STATE Numerically solving the general linear system $A_hu_h=b_h$ by V-cycle MGM
\end{algorithmic}
\begin{algorithmic}[1]
\STATE Pre-smooth:\quad $u_h:=\mathtt{smooth}^{\nu_1}(A_h,u_0,b_h)$
\STATE Get residual:\quad $r_h=b_h-A_hu_h$
\STATE Coarsen:\quad $r_H=I_h^Hr_h$
\IF{$H==h_0$}
\STATE Solve:\quad $A_H\xi_H=r_H$
\ELSE
\STATE Recursion:\quad $\xi_H=\mbox{V-cycle}(A_H,0,r_H)$
\ENDIF
\STATE Correct:\quad $u_h:=u_h+I_H^h\xi_H$
\STATE Post-smooth:\quad $u_h:=\mathtt{smooth}^{\nu_2}(A_h,u_h,b_h)$
\STATE Return $u_h$
\end{algorithmic}
\end{algorithm}

%\begin{algorithm}
%\caption{Stopping criterion for MGM}
%\label{Stopping criterion}
%\begin{algorithmic}[1]
%\STATE $u_h:=u_0$
%\STATE $r_0:=||A_hu_0-b_h||_2$
%\STATE $r_h:=r_0$
%\WHILE{$\frac{r_h}{r_0}>\epsilon$}
%\STATE $u_h:=\mbox{V-cycle}(A_h,u_h,b_H)$
%\STATE $r_h:=||A_hu_h-b_h||_2$
%\ENDWHILE
%\STATE Return $u_h$
%\end{algorithmic}
%\end{algorithm}

\renewcommand{\algorithmicrequire}{ \textbf{Input:}} %Use Input in the format of Algorithm
\renewcommand{\algorithmicensure}{ \textbf{Output:}} %UseOutput in the format of Algorithm

\begin{algorithm}
\caption{Calculating the Left and Right Fractional Derivatives}
\label{AFD}
\begin{algorithmic}[1]

\REQUIRE ~~\\ %算法的输入参数：Input
Original function $G(x) \in C^2(a,b)\cap C_0^1(a,b)$ and $\alpha \in (1,2)$\\
%Fractional derivative order $\alpha \in (1,2)$\\
\ENSURE ~~\\
Denote the values of numerically calculating $_aD_x^\alpha G(x)$ and $_xD_b^\alpha G(x)$ by $v_l$ and $v_r$\\
The algorithm JacobiGL of generating the nodes and weights of Gauss-Labatto integral with the weighting function $(1-x)^{1-\alpha}$ or $(1+x)^{1-\alpha}$  can be seen in\cite{Funaro:93,Hesthaven:07}\\
\STATE $z,w:=$JacobiGL$(1-\alpha,0,20)$
\STATE $v_l:=\frac{1}{\Gamma(2-\alpha)}\left(\frac{x-a}{2}\right)^{2-\alpha}
\sum\limits_{i=1}^{20} \frac{\partial^2 G}{\partial x^2}\left(\frac{x-a}{2}z_i+\frac{x+a}{2}\right)w_i$
\STATE $z,w:=$JacobiGL$(0,1-\alpha,20)$
\STATE $v_r:=\frac{1}{\Gamma(2-\alpha)}\left(\frac{b-x}{2}\right)^{2-\alpha}
\sum\limits_{i=1}^{20} \frac{\partial^2 G}{\partial x^2}\left(\frac{b-x}{2}z_i+\frac{b+x}{2}\right)w_i$
\end{algorithmic}
\end{algorithm}

\end{document}